\documentclass[11pt,a4paper]{siamart190516}

\usepackage[utf8]{inputenc}
\usepackage[english]{babel}

\usepackage{tikz}
\usetikzlibrary{shapes.geometric}

\usepackage{graphicx}
\usepackage{fancyhdr}
\usepackage{dsfont}
\usepackage{stmaryrd}

\usepackage{amsfonts}
\usepackage{amssymb}

\usepackage{amsmath}

\usepackage{amsfonts}
\usepackage{amssymb}
\usepackage{mathtools}
\usepackage{enumitem}

\usepackage{mathrsfs}
\usepackage{overpic} 

\usepackage{algorithm}
\usepackage[noend]{algpseudocode}
\algblock{ParFor}{EndParFor}
\algnewcommand\algorithmicparfor{\textbf{parallel for}}
\algnewcommand\algorithmicpardo{\textbf{do}}
\algrenewtext{ParFor}[1]{\algorithmicparfor\ #1\ \algorithmicpardo}
\algtext*{EndParFor}

\usepackage{mathtools,booktabs}
\DeclarePairedDelimiter\ceil{\lceil}{\rceil}
\DeclarePairedDelimiter\floor{\lfloor}{\rfloor}
\algrenewcommand\algorithmicrequire{\textbf{Input:}}
\algrenewcommand\algorithmicensure{\textbf{Output:}}

\usepackage{varwidth}

\usepackage[title]{appendix}

\newcommand{\rank}{{\rm rank}}

\usepackage{cite}
\usepackage{color}

\title{Parallel Algorithms for computing the Tensor-train decomposition\thanks{Submitted to the editors \today.
\funding{This work is partially supported by the National Science Foundation grants DMS-1818757, DMS-1952757 and DMS-2045646, and United States Department of Energy DE-AC02-05CH11231. This material is based upon work supported
by the National Science Foundation Graduate Research Fellowship under Grant No. DGE-1650441.}}}
\author{Tianyi Shi\thanks{Center for Applied Mathematics, Cornell University, Ithaca, NY 14853. (\email{ts777@cornell.edu})}
\and Maximilian Ruth\thanks{Center for Applied Mathematics, Cornell University, Ithaca, NY 14853. (\email{mer335@cornell.edu})}
\and Alex Townsend\thanks{Department of Mathematics, Cornell University, Ithaca, NY 14853. (\email{townsend@cornell.edu})}}
\headers{Parallel Tensor-train decomposition}{T. Shi, M. Ruth, and A. Townsend}

\begin{document}
\newcommand{\R}[0]{\mathbb{R}}
\newcommand{\C}[0]{\mathbb{C}}
\maketitle

\begin{abstract}
The tensor-train (TT) decomposition expresses a tensor in a data-sparse format used in molecular simulations, high-order correlation functions, and optimization. In this paper, we propose four parallelizable algorithms that compute the TT format from various tensor inputs: (1) Parallel-TTSVD for traditional format, (2) PSTT and its variants for streaming data, (3) Tucker2TT for Tucker format, and (4) TT-fADI for solutions of Sylvester tensor equations. We provide theoretical guarantees of accuracy, parallelization methods, scaling analysis, and numerical results. For example, for a $d$-dimension tensor in $\mathbb{R}^{n\times\dots\times n}$, a two-sided sketching algorithm PSTT2 is shown to have a memory complexity of $\mathcal{O}(n^{\floor{d/2}})$, improving upon $\mathcal{O}(n^{d-1})$ from previous algorithms.
\end{abstract}



\section{Introduction}\label{sec:introduction}
Multidimensional problems in a wide variety of applications have data or solutions that are often represented by tensors~\cite{kolda2009tensor}. A general tensor $\mathcal{X} \in \C^{n_1\times\cdots \times n_d}$ requires $\prod_{j=1}^d n_j$ degrees of freedom to store, which scales exponentially with the order $d$. Therefore, it is often essential to approximate or represent large tensors with data-sparse formats so that storing and computing with them is feasible. The tensor-train (TT) decomposition~\cite{oseledets2011tensor} is a tensor format with a storage cost that can scale linearly in $n_j$ and $d$. The TT format is used in molecular simulations~\cite{savostyanov2014exact}, high-order correlation functions~\cite{kressner2015low}, and partial differential equation (PDE) constrained optimization~\cite{dolgov2017low,benner2020low}. In practice, one tries to replace $\mathcal{X}$ by a tensor $\tilde{\mathcal{X}}$ with a data-sparse TT format such that
\begin{equation} 
\| \mathcal{X} - \tilde{\mathcal{X}} \|_F \leq \epsilon \| \mathcal{X} \|_F, \qquad \|\mathcal{X}\|_F^2 = \sum_{i_1=1}^{n_1} \cdots \sum_{i_d = 1}^{n_d} |\mathcal{X}_{i_1,\ldots,i_d}|^2,
\label{eq:FrobeniusNorm}
\end{equation} 
where $0\leq\epsilon<1$ is an accuracy tolerance~\cite{grasedyck2013literature,hackbusch2012tensor}. One major challenge that we address in this paper is how to compute $\tilde{\mathcal{X}}$ in a TT format from large $\mathcal{X}$ in parallel with a limited memory footprint. Once we obtain $\tilde{\mathcal{X}}$ in the TT format, tensor operations including addition, mode-$k$ products (see~\cref{eq:kfold}), contraction, and recompression can be executed in parallel as well. 

Unlike in the matrix case where the truncated singular value decomposition (SVD) provides the best rank-$k$ approximation, tensors admit various low-rank formats with different desired properties. Other than the TT format, which represents each entry as the product of a sequence (``train") of matrices, there is the canonical polyadic (CP) format that expresses a tensor as the sum of vector outer products~\cite{kolda2009tensor}, and the orthogonal Tucker format that ensures factor matrices have orthonormal (ON) columns~\cite{de2000multilinear}. There are also multiple hierarchical formats such as tree-Tucker~\cite{oseledets2009breaking} and quantized TT (QTT)~\cite{dolgov2012fast} that can capture latent data structures. The TT format is popular because of its connection to linear matrix algebra, enabling rigorous analysis and numerically accurate and stable algorithms.

Researchers have designed parallel tensor algorithms to exploit modern computing architectures and handle larger tensors emerging in applications. There are parallel algorithms for computing CP~\cite{li2017model,smith2015splatt}, Tucker, and hierarchical Tucker decomposition~\cite{austin2016parallel,grasedyck2019parallel,ballard2020tuckermpi,kaya2016high}. Subsequent operations can also be done in parallel in various tensor formats, especially tensor contractions~\cite{solomonik2014massively}, and operations in TT format~\cite{daas2020parallel}. However, despite some current work based on hierarchical tree structure~\cite{grigori2020parallel}, regularized least squares problem satisfied by each core~\cite{chen2017parallelized}, and multiple SVDs on tensor slices~\cite{wang2020adtt}, parallel TT decomposition has received less attention, perhaps, due to the sequential nature of TTSVD~\cite{oseledets2011tensor}. 

In this paper, we show that the column spaces of tensor unfoldings (see~\cref{sec:notation}) are connected by the TT cores (see~\cref{sec:TT}). Using this property, we develop new parallel algorithms that are scalable, stable, and accurate to compute TT formats of tensors. In particular, we distribute tensor information across several processors and ask each of them to contribute to computing the TT cores. We design parallel algorithms for various tensor input types:
\begin{itemize}[leftmargin=*,noitemsep]
\item \textbf{Parallel-TTSVD}: Previous TT decomposition methods such as TTSVD~\cite{oseledets2011tensor} and TT-cross approximation~\cite{oseledets2010tt} are sequential algorithms that require the entire tensor as input. In each iteration, both algorithms find one TT core by decomposing a specific matrix and use this core to determine the matrix in the next iteration. Based on the fact that there is a connection between the column space of various reshapes of a tensor (see~\cref{sec:notation}), we design an algorithm to compute the TT cores simultaneously by computing an ON basis for the column space of each tensor unfolding via SVD.

\item \textbf{Parallel Streaming TT Sketching (PSTT)}: Since SVD in Parallel-TTSVD can be computationally expensive, we can use randomized linear algebra to find ON bases that approximate the column space of tensor unfoldings. This algorithm is inspired by matrix sketching~\cite{halko2011finding}, Tucker sketching~\cite{sun2020low}, randomized algorithms for CP and Tucker format~\cite{ma2021fast}, and TT sketching in a sequential manner~\cite{che2019randomized}. Sketching algorithms are ideal for streaming data, where it is infeasible to store the tensor in cache. We show a two-sided version, PSTT2, has a storage cost as low as $\mathcal{O}(n^{\floor{d/2}})$. Moreover, PSTT2-onepass, a one-pass variant of PSTT2, uses only a single evaluation of each tensor entry, and is the most efficient in numerical experiments.

\item \textbf{TT2Tucker} and \textbf{Tucker2TT}: An ON basis of the column space of the second unfolding of each TT core allows us to get a Tucker format of the given tensor fast~\cite{batselier2020meracle}. Conversely, given a tensor in Tucker format, we can obtain its TT cores through the Tucker factor matrices and the TT cores of its Tucker core.

\item \textbf{TT-fADI}: Tensors also arise as the solutions of Sylvester tensor equations, i.e.,
\begin{equation} 
\mathcal{X} \times_1 A^{(1)} + \cdots + \mathcal{X} \times_d A^{(d)} = \mathcal{F}, \qquad A^{(k)} \in \mathbb{C}^{n_k\times n_k}, \quad \mathcal{F} \in\mathbb{C}^{n_1\times \cdots \times n_d},
\label{eq:TensorDisplacement} 
\end{equation} 
where `$\times_k$' denotes the $k$-mode matrix product of a tensor (see~\cref{eq:kfold})~\cite{shi2021compressibility}. If $\mathcal{F}$ is provided in its TT format, then we can find $\mathcal{X}$ in TT format via the factored alternating direction implicit (fADI) method that solves Sylvester matrix equations~\cite{benner2009adi}.

\item \textbf{Implementations in message passing interface (MPI)}: We implement our algorithms in a distributed memory framework using OpenMPI in C. Each process is responsible for streaming part of the tensor and storing part of the intermediate calculations. We use well-established linear algebra packages to optimize our codes, including matrix multiplications in BLAS3, and QR and SVD in LAPACKE.
\end{itemize}

\Cref{sec:background} reviews some necessary tensor notations, TT and orthogonal Tucker format, and Sylvester equations. In~\cref{sec:parallelTT}, we consider computing the TT decomposition of a given tensor in parallel, where we have access to any entry. Then, we provide scalability and complexity analysis of our algorithms and demonstrate their performance on synthetic datasets in~\cref{NumericalExamples}. Finally, in~\cref{sec:TTsylv}, we obtain the TT format of implicitly known tensors, given as solutions of Sylvester tensor equations.

\section{Tensor notations, tensor formats, and Sylvester equations} \label{sec:background}
In this section, we review some tensor notations, TT and orthogonal Tucker format for low rank tensor approximations, and Sylvester matrix equations. Throughout this paper, for a tensor $\mathcal{X}$, we look for an approximation $\tilde{\mathcal{X}}$ that has low tensor ranks and satisfies~\cref{eq:FrobeniusNorm} for some $0 < \epsilon < 1$.

\subsection{Tensor notation} \label{sec:notation}
We use upper class letters to represent matrices and calligraphic capital letters to represent tensors. We commonly use MATLAB-style notation ``:" for indices, i.e., $a\!:\!b$ represents the index set $\{a,a+1,\ldots,b\}$, and a single ``:" stands for all the indices in that dimension. For example, $A(:,3\!:\!4)$ or $A_{:,3:4}$ represents the submatrix of $A$ that contains its third and fourth columns, and $\mathcal{Y}(:,j,:)$ represents the matrix slice of the tensor $\mathcal{Y}$ by fixing the second index to be $j$. We also use $\mathcal{Y}(:)$ to stack all the entries of $\mathcal{Y}$ into a single vector using column-major ordering. We use the MATLAB command ``reshape" to reorganize elements of a tensor. If $\mathcal{Y} \in \C^{n_1 \times n_2 \times n_3}$, then ${\rm reshape}(\mathcal{Y},n_1n_2,n_3)$ returns a matrix of size $n_1n_2 \times n_3$ formed by stacking entries according to their multi-index. Therefore, $\mathcal{Y}(:)$ and ${\rm reshape}(\mathcal{Y},n_1n_2n_3,1)$ are equivalent. Similarly, if $Z \in \C^{n_1 n_2 \times n_3}$, then ${\rm reshape}(Z,n_1,n_2,n_3)$ returns a tensor of size $n_1 \times n_2 \times n_3$. Throughout, we use the notation for tensors found in~\cite{kolda2009tensor}, which we briefly review now for readers' convenience. Consider a tensor $\mathcal{X}\in\C^{n_1\times\cdots\times n_d}$, then we have the following definitions.

\begin{description}[leftmargin=*,noitemsep]
\item[Flattening by reshaping.]
One can reorganize the entries of a tensor into a matrix without changing the column-major ordering, and this idea is fundamental to the TT decomposition. The $k$th unfolding of $\mathcal{X}$ is represented as
\[X_k={\rm reshape}\left(\mathcal{X},\prod_{s=1}^k n_s,\prod_{s=k+1}^d n_s\right). \]

\item[Flattening via matricization.]
Another way to flatten a tensor is to arrange the mode-$n$ fibers to be the columns of a matrix~\cite{kolda2006multilinear}, and this operation is central for the orthogonal Tucker decomposition. We denote the $k$th matricization of $\mathcal{X}$ by $X_{(k)} \in \C^{n_k \times \prod_{j\neq k} n_j}$. Since mode-1 fibers are the tensor columns, we have $X_{(1)}=X_1$.

\item[The $k$-mode product.] The $k$-mode product of $\mathcal{X}$ with a matrix $A\in\C^{m \times n_k}$ is denoted by $\mathcal{X} \times_k A$, and defined elementwise as
\begin{equation}
(\mathcal{X} \times_k A)_{i_1,\ldots,i_{k-1},j,i_{k+1},\ldots,i_d} = \sum_{i_k = 1}^{n_k} \mathcal{X}_{i_1,\ldots,i_d}A_{j,i_k}, \quad 1 \le j \le m.
\label{eq:kfold} 
\end{equation} 
This is equivalent to computing $AX_{(k)}$ and reorganizing back to a tensor.

\end{description} 

\subsection{Tensor-train format} \label{sec:TT}
The TT format represents each tensor entry as the product of a sequence of matrices. A tensor $\mathcal{X}\in\C^{n_1\times \cdots \times n_d}$ has TT cores $\mathcal{G}_k \in \C^{s_{k-1} \times n_k \times s_k}$ for $1 \le k \le d$, if the cores satisfy
\[
\mathcal{X}_{i_1,\ldots,i_d} = \mathcal{G}_1(:,i_1,:)\mathcal{G}_2(:,i_2,:) \cdots \mathcal{G}_d(:,i_d,:), \qquad 1\leq i_k \leq n_k.
\]
Since the product of the matrices must be a scalar, we have $s_0 = s_d = 1$. We call $\pmb{s} = (s_0,\ldots,s_d)$ the size of the TT cores, and it is an entry-by-entry bound on the TT rank $\pmb{r}= (r_0,\ldots,r_d)$. In this way, a TT representation with TT core size $\pmb{s}$ requires $\sum_{k=1}^d s_{k-1}s_k n_k$ degrees of freedom for storage, which is linear in mode size $\pmb{n}=(n_1,\ldots,n_d)$ and order $d$.~\cref{fig:TT} illustrates a TT format with TT core size $\pmb{s}$.
\begin{figure} 
\centering 
\begin{tikzpicture}
\filldraw[black] (0,-0.5) node {$\mathcal{X}_{i_1,\ldots,i_d}$};
\filldraw[black] (1,-0.5) node {$=$};
\filldraw[color=black,fill=gray!20] (1.5,0) rectangle (4,-.5);
\filldraw[black] (2.8,-0.25) node {$\mathcal{G}_1(i_1,:)$};
\filldraw[black] (2.8,0.3) node {$1\! \times \! s_1$};
\filldraw[color=black,fill=gray!20] (4.2,0) rectangle (7.1,-2.5);
\filldraw[black] (5.7,-1.3) node {$\mathcal{G}_2(:,i_2,:)$};
\filldraw[black] (5.7,0.3) node {$s_1 \! \times \! s_2$};
\filldraw[black] (7.5,-1) node {$\cdots$};
\filldraw[color=black,fill=gray!20] (7.9,0) rectangle (10.3,-2.3);
 \filldraw[black] (9.1,-1) node {$\mathcal{G}_{d-1}(:,i_{d-1},:)$};
\filldraw[black] (9.1,0.3) node {$s_{d-2} \! \times \! s_{d-1}$};
\filldraw[color=black,fill=gray!20] (10.5,0) rectangle (11,-2.4);
\filldraw[black] (10.75,-1) node {\rotatebox{270}{$\mathcal{G}_{d}(:,i_{d})$}};
\filldraw[black] (10.75,0.3) node {$s_{d-1} \! \times \! 1$};
\end{tikzpicture}
\caption{The TT format with TT core size $\pmb{s} = (s_0,\ldots,s_d)$. Each entry of a tensor is represented by the product of $d$ matrices, where the $k$th matrix in the ``train" is selected based on the value of $i_k$.}
\label{fig:TT}
\end{figure}
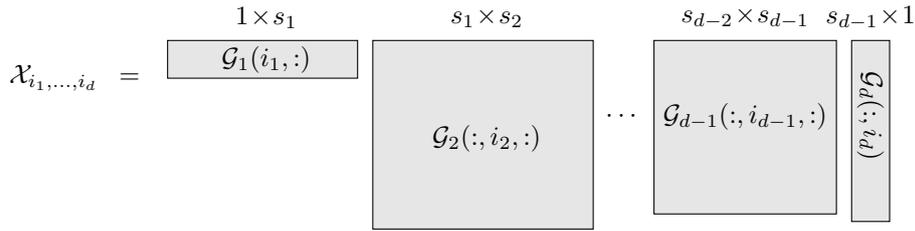 

The TTSVD algorithm computes a TT format by sequentially constructing the TT cores~\cite{oseledets2011tensor}. In this way, we can use ranks of the tensor unfoldings to bound entries of the TT rank~\cite{oseledets2011tensor}. That is,
\begin{equation} \label{eq:TT_trivial}
r_k \le {\rm rank}(X_k), \quad 1 \le k \le d-1,
\end{equation}
where $\rank(X_k)$ is the rank of the unfolding matrix $X_k$. Therefore, if the ranks of all the matrices $X_k$ for $1\leq k\leq d-1$ are small, then the TT format of $\mathcal{X}$ is data-sparse. In particular, if SVDs in TTSVD are truncated to have an accuracy of $\epsilon/\sqrt{d-1}$ in the Frobenius norm, then we obtain an approximation $\tilde{\mathcal{X}}$ that satisfies~\cref{eq:FrobeniusNorm} in the TT format. However, TTSVD is sequential, so we wish to further exploit~\cref{eq:TT_trivial} to compute a TT format in parallel.

\subsection{Orthogonal Tucker format} \label{sec:OrthogonalTucker}
The orthogonal Tucker format represents a tensor $\mathcal{X} \in \C^{n_1\times\cdots \times n_d}$ with a core tensor $\mathcal{G} \in\C^{t_1 \times \cdots \times t_d}$ and a set of factor matrices $A_1,\ldots,A_d$ with orthonormal columns~\cite{kolda2009tensor,de2000multilinear}:
\begin{equation}
\mathcal{X} = \mathcal{G} \times_1 A_1 \cdots \times_d A_d , \qquad A_k \in \mathbb{R}^{n_k\times t_k}.
\label{eq:Tucker} 
\end{equation}
In this case, we call $\pmb{t} = (t_1,\ldots,t_d)$ the size of the factor matrices of $\mathcal{X}$ and it provides an entry-by-entry bound on the multilinear rank $\pmb{\ell} = (\ell_1,\ldots,\ell_d)$. Such a decomposition contains $\sum_{k=1}^{d} n_k t_k + \prod_{k=1}^{d} t_k$ degrees of freedom, which is linear in size $\pmb{n}=(n_1,\ldots,n_d)$, and still exponential in dimension $d$ and thus can be infeasible for large $d$. Nevertheless, the Tucker format is very useful when each entry $t_j$ is significantly smaller than the corresponding mode size $n_j$. 

Higher-order SVD (HOSVD)~\cite{de2000multilinear} can be used to compute the orthogonal Tucker format of a given tensor. The algorithm utilizes an ON basis of each tensor matricization as the corresponding factor matrix and computes the core tensor with these matrices. Therefore, finding the factor matrices in parallel is easy, as the matricizations are independent and can be handled simultaneously. In terms of accuracy, if the factor matrices are calculated via SVDs with $\epsilon/\sqrt{d}$ accuracy in the Frobenius norm and $0 < \epsilon < 1$, then we obtain an approximation $\tilde{\mathcal{X}}$ that satisfies~\cref{eq:FrobeniusNorm} in the orthogonal Tucker format.

\subsection{Sylvester matrix equations and fADI}
A Sylvester matrix equation for an unknown matrix $W$ has the form
\begin{equation} \label{2d:sylv}
A_1W-WA_2^T = F, \quad A_1 \in \C^{m \times m}, \ A_2 \in \C^{n \times n}, \ F \in \C^{m \times n}.
\end{equation}
For simplicity, we assume $A_1$ and $A_2$ are normal matrices, then $W$ has a unique solution if the spectra of $A_1$ and $A_2$ are disjoint~\cite{simoncini2016computational}. When $F$ has a low rank factorization $F = UV^*$ with $U \in \C^{m \times r}$ and $V \in \C^{n \times r}$, we can use the factored alternating direction implicit (fADI) method~\cite{benner2009adi} to obtain $W$ also in low rank format by solving a sequence of shifted linear systems. The main takeaway from the fADI method is that it solves for the column space and row space of $W$ independently. The shifts used in the iterations are known in many situations~\cite{fortunato2020fast,townsend2018singular}. For example, one set of shift parameters $\pmb{p}$ and $\pmb{q}$ can be chosen as the zeros and poles of a rational function $r \in \mathcal{R}_{k,k}$, that can achieve a quasi-optimal Zolotarev number~\cite{zolotarev1877application}
\[
Z_k(\Lambda(A_1),\Lambda(A_2)) := \inf_{r \in \mathcal{R}_{k,k}} \frac{\sup_{z \in \Lambda(A_1)} |r(z)|}{\inf_{z \in \Lambda(A_2)} |r(z)|},\qquad k\geq 0,
\]
where $\Lambda(A_1)$ and $\Lambda(A_2)$ are the spectra of $A_1$ and $A_2$, and $\mathcal{R}_{k,k}$ is the set of rational functions of the form $s(x)/t(x)$ with polynomials $s$ and $t$ of degree at most $k$. This choice is closely related to the fact that Zolotarev numbers can be used to bound approximations of $W$ that satisfies~\cref{2d:sylv}~\cite{shi2021compressibility}; namely,
\begin{equation} \label{zolo_fro}
\|W-W_k\|_F \le Z_k(\Lambda(A_1),\Lambda(A_2)) \|W\|_F,
\end{equation}
where $W_k$ is the best rank-$k$ approximation of $W$.

\section{Parallel TT approximations from other tensor formats} \label{sec:parallelTT}
In this section, we focus on describing parallel algorithms to compute a TT approximation of a tensor $\mathcal{X}$ when we have access to all its entries. We consider three scenarios: (1) we can afford to store the whole tensor in cache (see~\cref{sec:parallelTTSVD}), (2) we can only afford to store a proportion of its entries in cache (see~\cref{sec:TTsketching}), and (3) the Tucker format of $\mathcal{X}$ is known and can be stored in cache (see~\cref{sec:TTTucker}).

\subsection{Parallel TT decomposition with SVD} \label{sec:parallelTTSVD}
The derivation of the parallel TT decomposition starts with the analogy between TTSVD and HOSVD. Roughly speaking, HOSVD follows a ``compress-then-combine" approach, which compresses all matricizations first and then computes the core tensor. This makes HOSVD for computing the Tucker decomposition naturally parallelizable. Comparatively, for the TT format, TTSVD has a sequential nature that alternates between reshaping and compressing. Here, we design a ``compress-then-combine" algorithm for computing a TT approximation. We compress all the unfoldings first and then combine the resulting matrices to obtain the TT cores. 

We first show that the ON bases of the column space of all tensor unfoldings are related for a $d$-dimensional tensor.
\begin{theorem} \label{thm:interlacing_d}
Let $\mathcal{X} \in \C^{n_1 \times \dots \times n_d}$, and $X_j \in \C^{\left(\prod_{i=1}^j n_i\right) \times \left(\prod_{i = j+1}^d n_i\right)}$ be its $j$th flattening for $1 \le j \le d-1$. If $X_j = U_jV_j^*$ with $U_j \in \C^{\left(\prod_{i=1}^j n_i\right) \times r_j}$, $V_j \in \C^{\left(\prod_{i = j+1}^d n_i \right) \times r_j}$, $r_j \le \min(\prod_{i=1}^j n_i,\prod_{i = j+1}^d n_i)$, and $U_j$ has ON columns, then for $1 \le k \le d-2$, there exists matrices $W_k \in \C^{r_k \times n_{k+1}r_{k+1}}$ such that 
\[ {\rm reshape}\left(U_{k+1}, \prod_{i=1}^k n_i ,n_{k+1}r_{k+1}\right) = U_kW_k. \]
\end{theorem}
\begin{proof}
To proceed with the proof, we use the fact that for $1 \le i \le d-2$, each column of $X_{i+1}$ consists of $n_{i+1}$ consecutive columns of $X_i$. This is true for any $d \ge 3$ so it suffices to show the statement holds when $\mathcal{X} \in \C^{n_1 \times n_2 \times n_3}$.

For notational simplicity, we denote the frontal slices of $\mathcal{X}$ by $X_f^{(j)} = \mathcal{X}(:,:,j) \in \C^{n_1 \times n_2}$ for $1 \le j \le n_3$ and the lateral slices by $X_\ell^{(k)} = \mathcal{X}(:,k,:) \in \C^{n_1 \times n_3}$ for $1 \le k \le n_2$. Then, by construction, we have
\[ X_1 = \begin{bmatrix} X_f^{(1)} & \cdots & X_f^{(n_3)} \end{bmatrix}, \quad X_2 = \begin{bmatrix}X_\ell^{(1)} \\[5pt]\vdots \\[5pt] X_\ell^{(n_2)} \end{bmatrix}. \]
The columns of the frontal slices and those of the lateral slices are columns of the tensor $\mathcal{X}$, so we can find the same column in the frontal and lateral slices. That is,
\[ \left(X_\ell^{(k)}\right)_j = \left(X_f^{(j)}\right)_k = \mathcal{X}(:,k,j), \quad 1 \le j \le n_3, \quad 1 \le k \le n_2, \]
where $\left(X_\ell^{(k)}\right)_j$ denotes the $j$th column of $X_\ell^{(k)}$, and similarly for $\left(X_f^{(j)}\right)_k$.
Given $X_1 = U_1V_1^*$, we can write $X_2$ as
\[ X_2 = \begin{bmatrix} U_1\left(V_1^{(1)}\right)^* \\[5pt] \vdots \\[5pt] U_1\left(V_1^{(n_2)}\right)^* \end{bmatrix} = (I \otimes U_1)Z^*, \]
where $V_1^{(i)}$ is a submatrix that contains columns $i, i+n_2,\dots,i+(n_3-1)n_2$ of $V_1$ for $1 \le i \le n_2$, and `$\otimes$' is the Kronecker product of two matrices.

Since $U_1$ has ON columns, $I \otimes U_1$ has ON columns, and $Z^* = ST^*$, where $S \in \C^{r_1n_2 \times r_2}$ has ON columns and $T \in \C^{n_3 \times r_2}$. Without loss of generality, we can assume that $U_2 = (I \otimes U_1)S$; otherwise, there is an orthogonal matrix $Y \in \C^{r_2 \times r_2}$ such that $U_2 = (I \otimes U_1)SY$, $T = TY$, and $V_2 = V_2Y$. By reshaping $U_2 = (I \otimes U_1)S$, we find
\[ {\rm reshape}(U_2, n_1,n_2r_2) = U_1W, \quad W = {\rm reshape}(S,r_1,n_2r_2), \]
which proves the statement for $d = 3$.
\end{proof}

One may notice that for $1 \le k \le d-2$, ${\rm reshape}(W_k, r_k, n_{k+1}, r_{k+1})$ is the size of the $(k+1)$st core in a TT format. In the next theorem, we show the accuracy of the approximation if we construct a tensor with TT cores given by ${\rm reshape}(W_k, r_k, n_{k+1}, r_{k+1})$.

\begin{theorem} \label{thm:perf_alg1}
Let $\mathcal{X} \in \C^{n_1 \times \dots \times n_d}$, and $0 < \epsilon < 1$. Suppose further that for $1 
\le j \le d-1$, each unfolding admits $\|X_j - U_jV_j^*\|_F \le \frac{\epsilon}{\sqrt{d-1}}\|\mathcal{X}\|_F$, where $U_j$ has ON columns. Then, the tensor $\tilde{\mathcal{X}}$ constructed by TT cores $\mathcal{G}_1 = U_1$, $\mathcal{G}_d = V_{d-1}^*$, and
\begin{equation}
\label{eq:core}
  \mathcal{G}_{k+1} = {\rm reshape}\left(U_k^* \ {\rm reshape}\left(U_{k+1}, \prod_{i=1}^k n_i ,n_{k+1}r_{k+1}\right), r_k, n_{k+1}, r_{k+1}\right),
\end{equation}
for $1 \le k \le d-2$, satisfies $\|\mathcal{X}-\tilde{\mathcal{X}}\|_F \le \epsilon\|\mathcal{X}\|_F$.
\end{theorem}
\begin{proof}
Since the TT cores are computed with $U_1,\dots,U_{d-1}$, we can express each element of $\tilde{\mathcal{X}}$ with the same matrices:
\begin{align} 
\tilde{\mathcal{X}}_{i_1,i_2\dots,i_d} &= \mathcal{G}_1(i_1,:) \mathcal{G}_2(:,i_2,:)\cdots \mathcal{G}_d(:,i_d) \nonumber \\
&= U_1(i_1,:)U_1^* U_2((i_2-1)n_1+1:i_2n_1,:) \cdots \nonumber \\
&\quad U_{d-2}^* U_{d-1}\left(\left((i_{d-1}-1)\prod_{k=1}^{d-2} n_k+1\right):\left(i_{d-1}\prod_{k=1}^{d-2}n_k\right),:\right)U_{d-1}^* X_{d-1}(:,i_d) \nonumber. 
\end{align}
Now, let $Y_{d-1} = U_{d-1}U_{d-1}^*X_{d-1}$, $\tilde{Y}_{j+1} = {\rm reshape}\left(Y_{j+1}, \prod_{k = 1}^{j} n_k, \prod_{k=j+1}^d n_k\right)$, and $Y_j = U_jU_j^*\tilde{Y}_{j+1}$ for $1 \le j \le d-2$. Then,
\begin{align}
\|\mathcal{X}-\tilde{\mathcal{X}}\|^2_F &= \|Y_1-X_1\|^2_F \nonumber \\
&= \|U_1U_1^*\tilde{Y}_2-X_1\|^2_F \nonumber \\
&= \|U_1U_1^*(\tilde{Y}_2-X_1)+(U_1U_1^*-I)X_1\|^2_F \nonumber \\
&= \|U_1U_1^*(\tilde{Y}_2-X_1)\|^2_F+\|(U_1U_1^*-I)X_1\|^2_F \nonumber \\
&\le \|\tilde{Y}_2-X_1\|^2_F+\|(U_1U_1^*-I)X_1\|^2_F \nonumber \\
&= \|Y_2-X_2\|^2_F+\|(U_1U_1^*-I)X_1\|^2_F, \nonumber
\end{align}
where the fourth equality holds since $(U_1U_1^*(\tilde{Y}_2-X_1))^*((U_1U_1^*-I)X_1) = 0$, the inequality holds since $U_1U_1^*$ is an orthogonal projection, and the last equality holds as reshaping preserves Frobenius norm. Following this argument, by induction on $\|Y_j-X_j\|_F^2$ for $1 \le j \le d-1$, we have
\begin{align}
\|\mathcal{X}-\tilde{\mathcal{X}}\|^2_F &\le \sum_{k=1}^{d-1} \|(I-U_kU_k^*)X_k\|_F^2 \label{eq:proj_acc} \\
&\le \sum_{k=1}^{d-1} \|(I-U_kU_k^*)(X_k-U_kV_k^*)\|_F^2 \nonumber \\
&= \sum_{k=1}^{d-1} \|X_k-U_kV_k^*\|_F^2 \le \epsilon^2\|\mathcal{X}\|_F^2. \nonumber
\end{align}
\end{proof}

\cref{thm:perf_alg1} provides an algorithm to compute a tensor $\tilde{\mathcal{X}}$ in TT format that approximates $\mathcal{X}$ (see~\cref{alg:1}). It is also simple to observe that~\cref{alg:1} can be performed in parallel, since the unfoldings of a tensor are independent.
\begin{algorithm}
\caption{Parallel-TTSVD: Given a tensor, compute an approximant tensor in TT format using SVD in parallel. }
\begin{algorithmic}[1]
\label{alg:1}
\Require {A tensor $\mathcal{X} \in \C^{n_1 \times \dots \times n_d}$ and a desired accuracy $0<\epsilon<1$}
\Ensure {TT cores $\mathcal{G}_1, \dots, \mathcal{G}_d$ of an approximant $\tilde{\mathcal{X}}$}
\For {$1 \le j \le d-1$}
\State Compute a rank $r_j$ approximation of the $j$th flattening of $\mathcal{X}$ in truncated SVD form so that $\|X_j - U_j\Sigma_jV_j^*\|_F \le \epsilon\|\mathcal{X}\|_F/\sqrt{d-1}$. 
\EndFor
\For {$1 \le k \le d-2$}
\State Calculate $W_{k+1} = U_k^* \ {\rm reshape}(U_{k+1}, \prod_{i=1}^k n_i ,n_{k+1}r_{k+1})$.
\State Set $\mathcal{G}_{k+1} = {\rm reshape}(W_{k+1}, r_k, n_{k+1}, r_{k+1})$.
\EndFor
\State Set $\mathcal{G}_1 = U_1$ and $\mathcal{G}_d = \Sigma_{d-1}V_{d-1}^*$.
\end{algorithmic}  
\end{algorithm}

Since the unfoldings have different sizes, the amount of work assigned to each processor in~\cref{alg:1} varies. Roughly speaking, processors that deal with $X_j$ when $j$ is close to $\floor{d/2}$ have the most computationally expensive SVD. In practice, one can replace the SVD in~\cref{alg:1} with the randomized SVD~\cite{halko2011finding}, in which case the computational complexity on each processor is $\mathcal{O}(r_j\prod_{i=1}^d n_i)$ where $r_j$ is the rank of the $j$th unfolding. In this scenario,~\cref{alg:1} is ideal when all the $r_j$'s are equal so that the computation is evenly distributed across all the processors.

\subsection{Parallel TT Sketching} \label{sec:TTsketching}
When one implements~\cref{alg:1} in a distributed computing environment, such as on a multi-threaded computer, a copy of the entire tensor $\mathcal{X}$ needs to be made on each processor. However, storing the tensor might not be feasible. Under this setting, we cannot use SVD as it requires all the tensor entries to be in cache. Instead, we may only be able to read a small portion of the entries of $\mathcal{X}$ at a time before discarding.

A common idea in this scenario for large matrices and tensors is sketching, where information about the matrix or tensor is obtained via matrix-vector multiplications. This idea is used for computing low-rank approximations of matrices~\cite{halko2011finding}, Tucker decomposition on tensors~\cite{sun2019low,ma2021fast}, and TT decomposition~\cite{che2019randomized}. In particular, SVDs in TTSVD can be replaced by sketching and a randomized range finder~\cite{che2019randomized}. Here, we develop a parallel TT sketching algorithm based on~\cref{thm:interlacing_d} (see~\cref{alg:2}). Since we want a truncated QR decomposition to reveal the rank of a given matrix, we use the column pivoted QR (CPQR)~\cite{chan1987rank}. This is a so-called ``two-pass" algorithm since $\mathcal{X}$ is used twice: the first time to compute an ON basis of the column space of each unfolding, and the second time to compute the last TT core.
\begin{algorithm}
\caption{PSTT: Given a tensor, compute an approximant tensor in TT format using sketching.}
\begin{algorithmic}[1]
\label{alg:2}
\Require {A tensor $\mathcal{X} \in \C^{n_1 \times \dots \times n_d}$, TT core size $\pmb{r}$, and an oversampling parameter $p$}
\Ensure {TT cores $\mathcal{G}_1,\ldots,\mathcal{G}_d$ of an approximant $\tilde{\mathcal{X}}$}
\For {$1 \le j \le d-1$}
\State Generate $\Phi_j \in \R^{\left(\prod_{k=j+1}^d n_k\right) \times (r_j+p)}$ with i.i.d. standard Gaussian entries. 
\State Calculate $S_j = X_j\Phi_j$, where $X_j$ is the $j$th flattening of $\mathcal{X}$
\State Compute a CPQR of $S_j$ to obtain $Q_j$ with ON cols and set $Q_j = Q_j(:,\!1\!:\!r_j)$.
\EndFor
\For {$1 \le k \le d-2$}
\State Calculate $W_{k+1} = Q_k^* \ {\rm reshape}(Q_{k+1}, \prod_{i=1}^k n_i ,n_{k+1}r_{k+1})$.
\State Set $\mathcal{G}_{k+1} = {\rm reshape}(W_{k+1}, r_k, n_{k+1}, r_{k+1})$.
\EndFor
\State Set $\mathcal{G}_1 = Q_1$, and $\mathcal{G}_d = Q_{d-1}^*X_{d-1}$.
\end{algorithmic}  
\end{algorithm}

Since $Q_j$ has ON columns for $1 \le j \le d-1$,~\cref{eq:proj_acc} provides an error bound for~\cref{alg:2}.
\begin{theorem} \label{thm:perf_sketch}
Let $\mathcal{X} \in \C^{n_1 \times \dots \times n_d}$, $\pmb{r}$ be the desired TT core size, and $p \ge 2$. The approximation $\tilde{\mathcal{X}}$ computed in~\cref{alg:2} satisfies
\begin{equation} \label{eq:sketch_bound}
\mathbb{E}\left[\|\mathcal{X}-\tilde{\mathcal{X}}\|_F^2\right] \le \sum_{j=1}^{d-1} \left(1+\frac{r_j}{p-1}\right) \left(\sum_{k = r_j+1}^{M_j} \sigma_k^2(X_j) \right), 
\end{equation}
where $\sigma_k(X_j)$ is the $k$th singular value of $X_j$ and $M_j = \min(\prod_{i=1}^j n_i, \prod_{i = j+1}^d n_i)$. Further assume that $p \ge 4$, then for all $u, t \ge 1$, 
\begin{equation} 
\|\mathcal{X}-\tilde{\mathcal{X}}\|_F^2 \le \sum_{j=1}^{d-1} \left( (1+t\sqrt{12r_j/p}) \left(\sum_{k = r_j+1}^{M_j} \sigma_k^2(X_j) \right)^{\tfrac{1}{2}} +ut\frac{e\sqrt{r_j+p}}{p+1}\sigma_{r_j+1}(X_j) \right)^2,
\label{eq:sketch_bound_prob} 
\end{equation} 
with failure probability at most $5t^{-p}+2e^{-u^2/2}$. \end{theorem}
\begin{proof}
The expectation bound in~\cref{eq:sketch_bound} follows from~\cref{eq:proj_acc} and~\cite[Thm 10.5]{halko2011finding}. The probability bound in~\cref{eq:sketch_bound_prob} follows from~\cref{eq:proj_acc} and~\cite[Thm 10.7]{halko2011finding}.
\end{proof}

If the accuracy in~\cref{thm:perf_alg1} is considered as a baseline, then~\cref{eq:sketch_bound} implies that the error is within a constant factor of the baseline for moderate $p$ with high probability. We can understand the probability bound as two parts. The sum of squares of the ``tail" singular values corresponds with the expected approximation error in~\cref{eq:sketch_bound}. Analogously, the $(j+1)$st singular value of each unfolding in the second term is related to the deviation above the mean.

A bottleneck of~\cref{alg:2} is the size of the dimension reduction maps (DRMs) $\Phi_j$, which grow exponentially with $d$. In practice, we can substitute them with Khatri-Rao products of smaller DRMs~\cite{sun2018tensor}. In other words, for $1 \le j \le d-1$, instead of using $\Phi_j \in \R^{\prod_{k=j+1}^d n_k \times (r_j+p)}$ with independent and identically distributed (i.i.d.) standard Gaussian entries, we use $\Psi_{d}^{(j)} \odot \cdots \odot \Psi_{j+1}^{(j)}$, where `$\odot$' denotes the Khatri-Rao product, and $\Psi_k^{(j)} \in \R^{n_k \times (r_j+p)}$ has i.i.d. standard Gaussian entries for $j+1 \le k \le d$. Then,
\begin{equation}
\label{eq:sketch_orig}
  X_j\left( \Psi_{d}^{(j)} \odot \cdots \odot \Psi_{j+1}^{(j)} \right) = \sum_{\ell_d=1}^{n_d} \cdots \sum_{\ell_{j+2}=1}^{n_{j+2}}(X_j)_{I(\ell_{j+1},\dots,\ell_d)} \Psi_{j+1}^{(j)} D_{\ell_{d}}\cdots D_{\ell_{j+2}},   
\end{equation}   
where $I(\ell_{j+1},\dots,\ell_d) = \sum_{k=j+3}^d (\ell_k-1)\prod_{p=j+2}^{k-1} n_p + (\ell_{j+2}-1)$, $X_j$ is partitioned as $X_j = \begin{bmatrix} (X_j)_1 & \cdots & (X_j)_{\prod_{k=j+2}^d n_k} \end{bmatrix}$ with $(X_j)_i \in \C^{\prod_{k=1}^j n_k \times n_{j+1}}$ for $1 \le i \le \prod_{k=j+2}^d n_k$, and $D_{\ell_q}$ is a diagonal matrix whose diagonal elements are the elements on row $\ell_q$ of $\Psi_q^{(j)}$ for $j+2 \le q \le d$.~\Cref{alg:2} with the Khatri-Rao DRMs gives slightly less accurate approximations, but the storage cost is only linear in $\pmb{n}$ and $d$.

When the tensor $\mathcal{X}$ is too large to access its elements for a second time in line 8 of~\cref{alg:2}, we design a ``one-pass" (or ``single-pass") algorithm that computes $\mathcal{G}_d$ without using $X_{d-1}$ directly~\cite{sun2020low}. To be specific, we generate a new dimension reduction map $\Psi_{d-1} \in \R^{\left(\prod_{k=1}^{d-1} n_k\right) \times (r_{d-1}+p)}$ with i.i.d. standard Gaussian entries and compute $T_{d-1} = \Psi_{d-1}^*X_{d-1}$ in lines 2 and 3 when $j = d-1$. Then,
\begin{equation}
\label{eq:onepass_final_core}
  T_{d-1} \approx \Psi_{d-1}^*Q_{d-1}Q_{d-1}^*X_{d-1} = (\Psi_{d-1}^*Q_{d-1})\mathcal{G}_d.   
\end{equation}
In this way, we have $\mathcal{G}_d \approx (\Psi_{d-1}^*Q_{d-1})^{\dagger}T_{d-1}$, where the pseudo-inverse exists with probability $1$. We use PSTT-onepass to denote this single-pass version of~\cref{alg:2}. The expected error of PSTT-onepass can also be determined.
\begin{theorem} \label{thm:perf_sketch_one}
Let $\mathcal{X} \in \C^{n_1 \times \dots \times n_d}$, $\pmb{r}$ be the desired TT core size, and $p \ge 2$. The approximation $\tilde{\tilde{\mathcal{X}}}$ computed by PSTT-onepass satisfies
\[
\begin{aligned}
\mathbb{E}\left[\|\mathcal{X}-\tilde{\tilde{\mathcal{X}}}\|_F^2\right] & \le \sum_{j=1}^{d-2} \left(\!1+\frac{r_j}{p-1}\right) \!\!\!\sum_{k = r_j+1}^{M_j} \sigma_k^2(X_j)
+\left(\!1+\frac{r_{d-1}}{p-1}\right)^2 \!\!\! \sum_{k = r_{d-1}+1}^{M_{d-1}} \sigma_k^2(X_{d-1}),  
\end{aligned}
\]
where $\sigma_k(X_j)$ is the $k$th singular value of $X_j$ and $M_j = \min(\prod_{i=1}^j n_i, \prod_{i = j+1}^d n_i)$ for $1 \le j \le d-1$.
\end{theorem}
\begin{proof}
Let $Z_{d-1} = Q_{d-1}(\Psi_{d-1}^*Q_{d-1})^{\dagger}\Psi_{d-1}^*X_{d-1}$, and for each $1\leq j\leq d-2$, let $\tilde{Z}_{j+1} = {\rm reshape}\left(Z_{j+1}, \prod_{k = 1}^{j-1} n_k, \prod_{k=j}^d n_k\right)$ and $Z_j = Q_jQ_j^*\tilde{Z}_{j+1}$. Then, from~\cref{eq:proj_acc} we find that
\[
\|\mathcal{X}-\tilde{\tilde{\mathcal{X}}}\|_F^2 \le \sum_{j=1}^{d-2} \|(I-Q_jQ_j^*)X_j\|_F^2+\|Z_{d-1}-X_{d-1}\|_F^2.
\]
So, we only need to bound the second term on the right hand side. Let $E_{d-1} = (\Psi_{d-1}^*Q_{d-1})^{\dagger}\Psi_{d-1}^*X_{d-1}$, then
\begin{align*}
\|Z_{d-1}&-X_{d-1}\|_F^2 = \|Q_{d-1}E_{d-1}-X_{d-1}\|_F^2 \\
&= \|Q_{d-1}Q_{d-1}^*Q_{d-1}E_{d-1}-Q_{d-1}Q_{d-1}^*X_{d-1}+Q_{d-1}Q_{d-1}^*X_{d-1}-X_{d-1}\|_F^2 \\
&= \|Q_{d-1}E_{d-1}-Q_{d-1}Q_{d-1}^*X_{d-1}\|_F^2+\|Q_{d-1}Q_{d-1}^*X_{d-1}-X_{d-1}\|_F^2 \\
&= \|E_{d-1}-Q_{d-1}^*X_{d-1}\|_F^2+\|(I-Q_{d-1}Q_{d-1}^*)X_{d-1}\|_F^2,
\end{align*}
where the third equality holds since $Q_{d-1}Q_{d-1}^*$ and $I-Q_{d-1}Q_{d-1}^*$ are orthogonal projectors. We are left to bound $\|E_{d-1}-Q_{d-1}^*X_{d-1}\|_F^2$. Plugging in $E_{d-1}$, we have
\begin{align*}
\|E_{d-1}&-Q_{d-1}^*X_{d-1}\|_F^2 = \|(\Psi_{d-1}^*Q_{d-1})^{\dagger}\Psi_{d-1}^*X_{d-1}-Q_{d-1}^*X_{d-1}\|_F^2 \\
&= \|(\Psi_{d-1}^*Q_{d-1})^{\dagger}\Psi_{d-1}^*X_{d-1}-(\Psi_{d-1}^*Q_{d-1})^{\dagger}(\Psi_{d-1}^*Q_{d-1})Q_{d-1}^*X_{d-1}\|_F^2 \\
&= \|(\Psi_{d-1}^*Q_{d-1})^{\dagger}\Psi_{d-1}^*(I-Q_{d-1}Q_{d-1}^*)X_{d-1}\|_F^2.
\end{align*}
The error bound follows from~\cite[Lemma B.1]{sun2020low}.
\end{proof}

The CPQR in line 4 of~\cref{alg:2} can be expensive when $j$ is large, and this still remains an issue when we use the ``single-pass" algorithm. It is simple to notice that~\cref{thm:interlacing_d} also holds for the row spaces of the unfoldings of $\mathcal{X}$. Let $d_* = \ceil{\frac{d}{2}}$. Then in practice, we compute $Q_j$, ON bases for the column spaces of $X_j$ when $j < d_*$, and $P_j$, ON bases for the row spaces of $X_j$ when $j > d_*$. In this way, the TT cores $\mathcal{G}_j$ for $j \neq d_*$ can be calculated similarly using lines 6 and 7 of~\cref{alg:2}, and the TT core in the middle $\mathcal{G}_{d_*}$ needs one extra step
\begin{equation}
\label{eq:PSTT2_final_core}
  \mathcal{G}_{d_*} = {\rm reshape}\left(\mathcal{X},\prod_{j=1}^{d_*-1}n_j, n_{d_*}, \prod_{j=d_*+1}^d n_j \right) \times_1 Q_{d_*-1}^* \times_3 P_{d_*+1}^*.   
\end{equation}
We call this variation PSTT2 to indicate that both column spaces and row spaces of tensor unfoldings are utilized, and the accuracy bounds in~\cref{thm:perf_sketch} continue to hold for PSTT2. Moreover, one can design PSTT2-onepass, a one-pass version of PSTT2, by carrying out an extra sketching step for the middle unfolding $X_{d_*}$. The sketching step can be performed on either the row or column space of $X_{d_*}$, and a pseudo-inverse follows it as in~\cref{eq:onepass_final_core} to obtain $\mathcal{G}_{d_*}$.

\subsection{Parallel TT and orthogonal Tucker conversion} \label{sec:TTTucker}
Some applications, including signal processing~\cite{de2004dimensionality}, computer vision~\cite{vasilescu2002multilinear}, and chemical analysis~\cite{henrion1994n}, construct and manipulate tensor data in the Tucker format. If one wants to explore latent structures in the TT format, a common approach is to convert the tensor back to the original format and then perform a TT decomposition. This method has a major drawback: it needs to explicitly construct the tensor in the original format, which ignores the low-rank structure one intends to utilize in both TT and Tucker format. Here, we develop a method to directly approximate a tensor $\mathcal{X}$ in orthogonal Tucker format by another tensor $\tilde{\mathcal{X}}$ in TT format.

If $\mathcal{X}$ has an orthogonal Tucker format~\cref{eq:Tucker}, and $\mathcal{H}_1,\dots,\mathcal{H}_d$ are the TT cores of $\mathcal{G}$, then $\mathcal{G} \times_j A_j$ is equivalent to $\mathcal{H}_j \times_2 A_j$ while keeping other cores unchanged for $1 \le j \le d$. In this way,~\cref{eq:Tucker} is equivalent to updating each TT core of $\mathcal{G}$ independently with the Tucker factor matrices. Therefore, we can use these updated cores as the TT cores of an approximation $\tilde{\mathcal{X}}$ (see~\cref{alg:3}).

\begin{algorithm}
\caption{Tucker2TT: Given the Tucker decomposition of a tensor, compute an approximant tensor in TT format.}
\begin{algorithmic}[1]
\label{alg:3}
\Require {The Tucker core $\mathcal{G}$ and factor matrices $A_1,\ldots,A_d$ of a tensor $\mathcal{X}$ (see~\cref{eq:Tucker})}
\Ensure {The TT cores $\mathcal{T}_1,\ldots,\mathcal{T}_d$ of an approximant tensor $\tilde{\mathcal{X}}$}
\State Perform a parallel TT decomposition on $\mathcal{G}$ and get TT cores $\mathcal{H}_1, \dots, \mathcal{H}_d$.
\For {$1 \le j \le d$}
\State Set $\mathcal{T}_j = \mathcal{H}_j \times_2 A_j$. 
\EndFor
\end{algorithmic}  
\end{algorithm}

As the Tucker core, $\mathcal{G}$ is much smaller in size than the original tensor $\mathcal{X}$, so line~$1$ of~\cref{alg:3} is computationally cheaper than forming $\mathcal{X}$ explicitly and computing a TT decomposition. In addition, the parallel TT decomposition to be used in line~$1$ depends on whether there is a prescribed accuracy of $0 < \epsilon < 1$ or a prescribed TT rank $\pmb{r}$. Finally, we can guarantee the algorithm's performance by showing that ${\rm rank}(X_j) = {\rm rank}(G_j)$ for $1 \le j \le d-1$.

\begin{lemma} \label{lm:tucker_tt_ranks}
Suppose a tensor $\mathcal{X} \in \C^{n_1 \times \dots \times n_d}$ has a Tucker decomposition~\cref{eq:Tucker}, where $A_j \in \C^{n_j \times s_j}$ and has ON columns for $1 \le j \le d$. Then, for $1 \le j \le d-1$, we have ${\rm rank}(X_j) = {\rm rank}(G_j)$.
\end{lemma}
\begin{proof}
We get the following equation by reshaping~\cref{eq:Tucker}:
\[ X_j = (A_j \otimes \dots \otimes A_1)G_j(A_d \otimes \dots \otimes A_{j+1})^T. \]
Since all $A_j$'s have ON columns, the rank of $X_j$ equals the rank of $G_j$.
\end{proof}

If one is provided with a tensor $\mathcal{X}$ in TT format, then it is also possible to find an approximation $\tilde{\mathcal{X}}$ in orthogonal Tucker format. By explicit calculation, one finds that $X_j$, the $j$th unfolding of the tensor $\mathcal{X}$ for $1 \le j \le d$, can be computed with the unfoldings of the TT cores:
\begin{equation} \label{eq:flattening_core_j}
X_j = \prod_{i = 1}^{j-1} \left(I_{\prod_{k=i+1}^j n_k} \otimes (G_i)_2\right) (G_j)_2 (G_{j+1})_1 \prod_{i = j+2}^d \left( (G_i)_1 \otimes I_{\prod_{k = j+1}^{i-1} n_k}\right),
\end{equation}
where $I_n$ is the identity matrix of size $n \times n$, and $(G_i)_p$ is the $p$th unfolding of $\mathcal{G}_i$ for $1 \le p \le 2$. Then, rewriting~\cref{eq:flattening_core_j} gives $X_j = P_j (G_j)_2 Q_j$ and
\begin{align*}
P_j = I_{n_j} \otimes \left( \prod_{i = 1}^{j-1} \left(I_{\prod_{k=i+1}^{j-1} n_k} \otimes (G_i)_2\right) \right), \quad Q_j = \prod_{i = j+1}^d \left( (G_i)_1 \otimes I_{\prod_{k = j+1}^{i-1} n_k}\right).
\end{align*}
We find that 
\begin{align} 
{\rm reshape}(\mathcal{X}, \prod_{k = 1}^{j-1} n_k, n_j, \prod_{k = j+1}^d n_k) = \mathcal{G}_j &\times_1 \left( \prod_{i = 1}^{j-1} \left(I_{\prod_{k=i+1}^{j-1} n_k} \otimes (G_i)_2\right) \right) \nonumber \\
&\times_3 \left(\prod_{i = j+1}^d \left( (G_i)_1 \otimes I_{\prod_{k = j+1}^{i-1} n_k}\right)\right).\nonumber
\end{align}
We can therefore find relationships between matricizations of $\mathcal{X}$ and those of $\mathcal{G}_j$:
\[
X_{(j)} = (G_j)_{(2)} \left[\left(\prod_{i = j+1}^d (G_i)_1 \otimes I_{\prod_{k = j+1}^{i-1} n_k}\right) \otimes \left( \prod_{i = 1}^{j-1} I_{\prod_{k=i+1}^{j-1} n_k} \otimes (G_i)_2 \right)^T \right],
\]
where $X_{(j)}$ is the $j$th matricization of $\mathcal{X}$ and $(G_j)_{(2)}$ is the second matricization of $\mathcal{G}_j$. When $(G_1)_2, \dots, (G_{j-1})_2$ have linearly independent columns and $(G_{j+1})_1,\dots,(G_d)_1$ have linearly independent rows, the column space of $X_{(j)}$ and $(G_j)_{(2)}$ match. This criterion can be satisfied when the TT rank of $\mathcal{X}$ is optimal. In this way, we can use any desired method, such as HOSVD or Tucker sketching, to find an orthogonal Tucker approximation of $\mathcal{X}$, simply by replacing $X_{(j)}$ by $(G_j)_{(2)}$. We summarize this procedure in~\cref{alg:4}, and use a particular version of HOSVD (see~\cite[Alg. 1]{batselier2020meracle}). 
\begin{algorithm}
\caption{TT2Tucker: Given a tensor in TT format, compute an approximant tensor in orthogonal Tucker format.}
\begin{algorithmic}[1]
\label{alg:4}
\Require {The TT cores $\mathcal{T}_1, \dots, \mathcal{T}_d$ of a tensor $\mathcal{X}$}
\Ensure {The TT cores $\mathcal{H}_1,\ldots,\mathcal{H}_d$ of the Tucker core $\mathcal{G}$, and factor matrices $A_1,\ldots,A_d$ of $\tilde{\mathcal{X}}$}
\For {$1 \le j \le d$}
\State Compute $A_j$ with ON columns that approximates column space of $(T_j)_{(2)}$.
\State Calculate $\mathcal{H}_j = \mathcal{T}_j \times_2 A_j^*$.
\EndFor
\end{algorithmic}  
\end{algorithm}

This algorithm is parallelizable since the TT cores of $\mathcal{X}$ are independent. Users need to either prescribe a desired accuracy or a multilinear rank to discover orthonormal bases of column spaces in step 2. HOSVD guarantees the performance of this algorithm if we use SVD, or~\cite[Thm. 5.1]{sun2020low} if we use sketching. Compared to HOSVD or Tucker sketching,~\cref{alg:4} is faster due to the known TT cores. As a result, memory costs and computation powers can be significantly reduced.

\section{Complexity Analysis and Numerical Examples} \label{NumericalExamples}
In this section, we model the computational and spatial cost of the proposed methods for TT decomposition, and compare these models with numerical experiments\footnote{For codes, see \url{https://github.com/SidShi/Parallel\_TT\_sketching}}. For simplicity, we assume that the tensor $\mathcal{X}$ of dimension $d$ is ``square,'' meaning that $n_i = n$ for all $1 \leq i \leq d$ and $r_i = r$ for all $ 1 \leq i \leq d-1$. 

We focus on the two-sided sketching methods PSTT2 and PSTT2-onepass defined at the end of \cref{sec:TTsketching}, since both the SVD based and the one-sided PSTT algorithms have poor spatial complexity. In addition, the SVD algorithm is much slower than the counterparts with sketching. As a baseline, we compare the results to a modified version of Algorithm 5.1 in~\cite{che2019randomized}, which is referred to in this manuscript as Serial Streaming TT Sketching (SSTT). 

In~\cref{sec:Computational_Details}, we discuss the computational environment and software used to implement our methods. Next, in~\cref{sec:Parallel_Tensor_Sketching}, we discuss how PSTT2, PSTT2-onepass, and SSTT are modified to be performed in a distributed memory environment. In~\cref{sec:Memory_Complexity}, we discuss the improvements of PSTT2 and PSTT2-onepass over previous serial algorithms in terms of spatial complexity. Finally, in~\cref{sec:Time_Complexity} we discuss time complexity of these algorithms. To illustrate the practicality of our algorithms, we show numerical experiments in~\cref{sec:Memory_Complexity} and~\cref{sec:Time_Complexity}.

\subsection{Computational Details} \label{sec:Computational_Details}

All experiments are performed in C, using the OpenMPI implementation of MPI for parallelization. All subroutines take advantage of LAPACKE and BLAS to vectorize large linear algebra operations, such as matrix-matrix multiplication and QR factorization. Experiments are performed on a machine with eight 12-core compute nodes, each consisting of two Intel Xeon E5-2620 v3 processors with 32 GB per node. For each trial, we measure the relative Frobenius error of the TT approximation, and achieve $\epsilon < 10^{-10}$ (see~\cref{eq:FrobeniusNorm}). For our experiments, we assume that the ranks are known apriori. Nevertheless, when they are not known, one can implement an adaptive algorithm such as Algorithm 5.2 in~\cite{che2019randomized}.

To measure the memory improvements of the two-sided methods, we compare the total allocated memory using the gperftools implementation of TCMalloc. In particular, the memory is measured on each core and then averaged. Transient memory allocations by MPI are ignored in the overall memory measurement. 

Throughout the following sections, we refer to some standard MPI functions to describe the parallel algorithm. These primarily include 
\begin{itemize} [leftmargin=*,noitemsep]
  \item \textit{Send:} the operation of sending some array of memory from one core to another.
  \item \textit{Receive:} the operation of receiving the memory sent by \textit{send}.
  \item \textit{Reduce:} the operation of summing matrices from a group of cores, used to summarize information gained by individual cores.
\end{itemize}
To avoid confusion in the following sections, we use `core' to refer to a single core of the CPU and `TT core' to refer to a ``train" in the TT format.

\subsection{Parallel Tensor Sketching} \label{sec:Parallel_Tensor_Sketching}
We first discuss the process of parallelizing SSTT, PSTT2, and PSTT2-onepass. We focus on the parallel sketching step, which is responsible for most of the time and storage. The predominance of the sketching step is emblematic of the ``compress-then-combine'' approach, as most of the computational work happens when compressing. The other steps needed for a fully parallel algorithm are discussed at the end of the section.

The parallel algorithms have a structure that is reminiscent of dense matrix-matrix multiplication algorithms such as Cannon's algorithm~\cite{cannon1969cellular} or SUMMA~\cite{vandegeijn1997summa}. These algorithms perform the operation $C = AB$ by dividing $A$, $B$, and $C$ into submatrices and distributing these submatrices across multiple cores. In this way, the overall spatial complexity is reduced.

For the parallel implementation of the TT sketching algorithms, a natural extension is to partition an unfolding $X_j$ into submatrices, and then proceed with standard matrix-matrix multiplications for the sketch step (see~\cref{alg:2})
\begin{equation}
\label{eq:easy_sketch}
  S_i = X_i \Phi_i.
\end{equation}
However, a computational hurdle with this approach is that we need to sketch multiple unfoldings during the same computational step since we want as few passes of the tensor as possible. As such, we aim for a guarantee that if $M$ is a submatrix of $X_i$, then some reshaping of $M$ is also a submatrix of $X_{i'}$ for some $i\neq i'$. In this way, parallel linear algebra algorithms work equally well for all unfoldings. 

For this reason, we introduce the notion of a \textit{sub-tensor}, which is a multilinear generalization of the submatrix. In each dimension $k$, we partition the tensor index $1\leq k_i\leq n$ into $P_i$ equal ``chunks'' (or in the case that $P_i$ does not divide into $n$, nearly equal chunks). In MATLAB notation, we have that $\mathcal{Y}_j$ is a sub-tensor of $\mathcal{X}$ if
\[
  \mathcal{Y}_j = \mathcal{X}\left(1 + \frac{n(j_1-1)}{P_1}:\frac{nj_1}{P_1}, 1+\frac{n(j_2-1)}{P_2}:\frac{nj_2}{P_2}, \dots, 1+\frac{n(j_d-1)}{P_d}:\frac{nj_d}{P_d} \right),
\]
where $\mathbf{j} = (j_1, \dots, j_d)$ is a multi-index specifying the target sub-tensor. Defining $P = \prod_{k = 1}^{d} P_k,$ it is clear that the memory needed to store $\mathcal{Y}_j$ is a factor of $1/P$ smaller than the memory needed to store $\mathcal{X}$. We also notice that any unfolding of $\mathcal{Y}_j$ is a submatrix of an unfolding of $\mathcal{X}$, although those submatrices are not necessarily contiguous. For notational simplity, we use the vectorized index notation $j = 1 + \sum_{k = 1}^{d} (j_k-1)\prod_{\ell=1}^{k-1} P_\ell$ for $1 \leq j \leq P$ as well. 

The sub-tensor gives an efficient method to store one part of the multiplication in~\cref{eq:easy_sketch}. To be specific, the matrix $\Phi_i$ can be stored in $d$ or fewer matrices of size $\mathcal{O}(nr)$ via the Khatri-Rao DRMs in~\cref{eq:sketch_orig}. This is a small enough cost that each core can store every $\Phi_i$. As a result, the only thing we need to distribute is the ``sketch'' $S_i$. In fact, the matrices $S_i$ dominate the memory complexity of PSTT2 and PSTT2-onepass (see~\cref{sec:Memory_Complexity}), and are therefore important to distribute. We distribute them by splitting $S_i$ into \textit{sub-sketches} $S_{i,k}$, with the column sub-sketches defined by
\begin{equation*}
  S_{i,k} = 
  {\rm reshape}(S_i, \underbrace{n, \dots, n}_{d-L}, r)
  \left(1 + \frac{n(k_1 - 1)}{P_{1}}:\frac{nk_1}{P_{1}}, \dots, 1+\frac{n(k_{d-i}-1)}{P_{d-i}}:\frac{nk_{d-i}}{P_{d-i}}, : \right).
\end{equation*}
We note that sketches for finding row spaces are performed by simply using $X_j^T$ instead of $X_j$ in \cref{eq:easy_sketch}. Row sub-sketches and the other necessary steps can be immediately derived from the same reasoning. 

Using the notions of sub-tensors and sub-sketches, we can expand \cref{eq:easy_sketch} into an explicit form with the Khatri-Rao DRMs $\Psi_\ell$ as $S_{i, k} = \sum_{j_{i+1}=1}^{P_{i+1}} \cdots \sum_{j_d=1}^{P_d} S_{i,k,j}$ with
\begin{align}
  S_{i, k, j} =& \mathrm{reshape}\left(\mathcal{Y}_{j}, \frac{n}{P_1}, \dots, \frac{n}{P_i}, \prod_{\ell=i+1}^{d} \frac{n}{P_\ell} \right) \times_{i+1}\\\nonumber
  &\left( \Psi_{d}^{(i)}\left(1+\frac{n(j_d-1)}{P_d}:\frac{nj_d}{P_d},:\right) \odot \cdots \odot \Psi_{i+1}^{(i)} \left(1 + \frac{n(j_{i+1}-1)}{P_{i+1}}:\frac{nj_{i+1}}{P_{i+1}},:\right)\right)^{T},\label{eq:KR_sub_product}
\end{align}
where $S_{i,k,j}$ is the contribution of the sub-tensor $\mathcal{Y}_j$ to the sub-sketch $S_{i,k}.$ 

Now, we write down the parallel sketching procedure, which is a subroutine (see~\cref{alg:parsketch}) for PSTT2, PSTT2-onepass, and SSTT. For generality, we assume that the input to each algorithm is some function $f$ that takes as input the index of the tensor $(\ell_1, \dots, \ell_d)$, and outputs a value of the tensor $\mathcal{X}_{\ell_1, \dots, \ell_d}$. This function allows us to load parts of the tensor into memory without necessarily loading the full tensor. We also assume that the time it takes to load a single element $\tau_f$ does not depend on the element or the number of elements loaded concurrently. In practice, this $f$ can be a function that reads data from a file or evaluates a known function. 

\Cref{alg:parsketch} outputs full sketches $S_i$ for multiple values of $i$, and each of the sketches is stored as sub-sketches across all cores. If there are $C$ cores, then each core is responsible for streaming approximately $P/C$ sub-tensors and stores roughly a factor of $1/C$ of the number of sub-sketches $S_{i,j}$. By default, we assume that the sub-sketches are distributed in ``column-major'' order, i.e., the cores each store sub-sketches $S_{i,j}$ that are consecutive in $j$. To contribute the sub-sketch to the core holding $S_{i,k}$, it is necessary to add an additional send/receive communication step. This algorithm is convenient for PSTT2 and PSTT2-onepass, as all sub-sketches can be computed with a single stream of a given sub-tensor.

\begin{algorithm}
  \caption{Find multiple sketches of a tensor in parallel.}
  \begin{algorithmic}[1]
  \label{alg:parsketch}
  \Require {Tensor oracle $f$, sketch dimensions $\mathbf{i}$, Khatri-Rao DRMs $\Psi^{(k)}_i$}
  \Ensure {Sketches $S_i$} for all $i\in\mathbf{i}$
  \State Initialize $S_{i,k}$ to zero for each $i\in\mathbf{i}$, distributed among the available cores
  \ParFor {$1 \leq j \leq P$}
    \State Load $\mathcal{Y}_j$ into memory via $f$
    \For {$i$ in $\mathbf{i}$}
      \State Compute $S_{i,k,j}$ via~\cref{eq:KR_sub_product}
      \State Send $S_{i,k,j}$ to owner of $S_{i,k}$
      \For {each $S_{i, k, j'}$ received}
        \State Add $S_{i, k, j'}$ to $S_{i,k}$
      \EndFor
    \EndFor
  \EndParFor
  \end{algorithmic}  
\end{algorithm}  

In all three algorithms, the step after~\cref{alg:parsketch} is to find an ON basis for $S_{i,k}$, except for the middle sketch of PSTT2-onepass. For this purpose, we use the skinny QR algorithm in~\cite{benson2013direct}, which is fast and has a low memory overhead in comparison to the sketching step. Furthermore, other steps that need to be parallelized are:
\begin{itemize}[leftmargin=*,noitemsep]
  \item \textbf{SSTT}: The first step is to find an ON basis for the column space using~\cref{alg:parsketch} with $\mathbf{i}=1$ and skinny QR, and use this basis as the first TT core $\mathcal{G}_1$. Because $G_1$ only has $nr$ entries, it can be stored in each core to reduce the overall required communication. Then, we compute $Z = \mathcal{G}_1^T X_1$ and obtain sub-sketches of $Z$ with~\cref{alg:parsketch}. These sketching and multiplication steps are repeated until all TT cores are obtained, and we distribute all sub-sketches among the cores. Overall, the second pass of streaming takes a significant amount of time, and the largest storage contribution comes from storing the first calculated $Z$. 
  \item \textbf{PSTT2}: To obtain all but the middle TT core, it is necessary to multiply ON bases of column/row spaces against each other via~\cref{eq:core}, which can be easily rewritten in terms of sub-tensors. Since the ON bases are much smaller than the tensor itself, communicating sub-sketches avoids high communication costs. In the end, the middle TT core is obtained via another streaming loop to perform~\cref{eq:PSTT2_final_core}, which adds significant computational time to the overall algorithm.
  \item \textbf{PSTT2-onepass}: All TT cores but the middle one are obtained as in PSTT2. Then, the middle TT core is computed with two matrix-matrix multiplications and a small least-squares problem. These are rewritten in terms of sub-sketches, and are cheap as they use already compressed data.
\end{itemize}

\subsection{Memory Complexity} \label{sec:Memory_Complexity}
In this section, we analyze the memory complexity of our algorithms. We give estimates of the memory costs of a sub-tensor, a TT representation, and the three algorithms SSTT, PSTT2, and PSTT2-onepass. The asymptotic costs are then compared to measured total memory allocation per core from numerical experiments, showing PSTT2 and PSTT2-onepass have lower overall memory requirements, especially for high-dimensional tensors.

Throughout the section, we focus on trials using the Hilbert tensor, defined as
\begin{equation*}
  \mathcal{X}_{i_1,\dots,i_d} = \frac{1}{1-d+i_1+\dots+i_d}, \quad 1 \le i_j \le n_j, \quad 1 \le j \le d.
\end{equation*}
It is known that this tensor can be accurately approximated by a numerically low TT rank tensor, and the TT ranks can be estimated a priori~\cite{shi2021compressibility}. Also, it is apparent that the total memory allocated does not depend on actual values of the tensor, but only on the dimension $d$, sizes $\mathbf{n}$, and ranks $\mathbf{r}$. Therefore, even though the Hilbert tensor is an artificial example, the memory results generalize to real-world tensors of similar sizes and ranks. We report numerical experiments for $d=3$, $d=5$, and $d=9$ Hilbert tensors with sizes $\mathbf{n}$, ranks $\mathbf{r}$, and partitions $\mathbf{P}$ given in~\cref{tab:1}.

\begin{table}
  \centering
  \begin{tabular}{c|c|c|c}
    $d$ & 3 & 5 & 9\\
    \hline 
    $\mathbf{n}$ & $960,960,960$ & $96,96,96,96,96$ & $12,12,12,12,12,12,12,12,12$\\
    $\mathbf{r}$ & $1,25,25,1$  & $1,17,18,18,17,1$ & $1,12,18,18,19,19,18,18,12,1$\\
    $\mathbf{P}$ & $96,1,96$   & $96,1,1,1,96$   & $12,6,1,1,1,1,1,6,12$\\
    tensor size & $6.59$ GB & $60.8$ GB & $38.4$ GB \\
    TT size & $4.94$ MB & $0.710$ MB & $0.197$ MB\\
    sub-tensor size & $0.769$ MB & $6.79$ MB & $7.60$ MB\\ 
  \end{tabular}
  \caption{\label{tab:1} Dimensions and memory sizes of the three tensors used for profiling. The `tensor size' line is calculated through the formula $8 n^d / 2^{30}$, and the `TT size' and `sub-tensor size' lines are measured via heap profiling. One can confirm that the sub-tensor sizes are nearly a factor of $1/P$ off from the tensor size, with discrepancies due to auxiliary information stored in the sub-tensor data structure.}
\end{table}

For each core, the number of stored entries of a sub-tensor and a TT format are
\[
  M_{\rm sub-tensor} = n^d/P, \qquad M_{\rm TT} = (d-2) r^2 n + 2 r n,
\]
where we use capital `$M$'s to denote all spatial costs. For the actual memory required, these quantities can be multiplied by the memory of a single entry. We note that neither of the above costs depend explicitly on the number of cores $C$, assuming $P>C$. The memory requirements for the three Hilbert examples are given in~\cref{tab:1}. We see that while the tensors have sizes in the gigabyte range, each of the cores only stores on the order of megabytes for the given values of $\mathbf{P}$. 

One can then express the memory costs of the individual algorithms as:
\begin{itemize}[leftmargin=*,noitemsep]
\item \textbf{SSTT}: As mentioned in~\cref{sec:Parallel_Tensor_Sketching}, the predominant memory cost comes from storing the first calculation $Z = \mathcal{G}_1^T X_1$, which has $r n^{d-1}$ entries. We distribute the entries among all cores, so the asymptotic memory cost of SSTT is
\begin{equation}
\label{eq:M_SSTT}
  M_{\rm SSTT} = \mathcal{O}(r n^{d-1} / C).
\end{equation}

\item \textbf{PSTT2}: The only major storage cost is to store the sketches. Each column sketch has a per-core memory cost of 
\begin{equation*}
  M_{S_i} = n^i r/C + (d-i) n r,
\end{equation*}
where the two terms are the memory needed for $S_i$ and the random Khatri-Rao DRMs. One needs to be careful that the above expression holds when $\mathbf{P}$ is large enough in the correct dimensions so that $S_j$ can be split into $C$ different sub-sketches. In practice, we choose $\mathbf{P}$ that the divisions are concentrated near $P_1$ and $P_d$ as in~\cref{tab:1}. This allows all column and row sketches to be distributed among the cores. We choose the middle index $d_* = \ceil{d/2}$ so that the storage is balanced among row and column sketches, and thus the overall storage of PSTT2 is
\begin{equation}
\label{eq:M_PSTT2}
  M_{\rm PSTT2} = \mathcal{O}\left(r n^{\floor{d/2}}/C + d r n \right).
\end{equation}
When $d>4$ and $C \ll n$, the first term dominates the second, improving upon SSTT by a factor of $n^{\ceil{d/2}-1}$. When $d = 3$, the second term -- the storage cost of the DRMs -- exceeds that of the actual sketch matrices. Nevertheless, this complexity is better than that of SSTT, which depends on $n^2$. 

\item \textbf{PSTT2-onepass}: The storage cost of PSTT2-onepass is different from that of PSTT2 only because of the final middle sketch. Hence, the complexity is
\begin{align}
\label{eq:M_PSTT2onepass}
  M_{\rm PSTT2-onepass} = \mathcal{O}\left(r n^{\ceil{d/2}}/C\right).
\end{align}
When $d$ is even, this spatial complexity is asymptotically the same as~\cref{eq:M_PSTT2}. When $d$ is odd, however, the two complexities differ. This is most apparent for $d=3$, when~\cref{eq:M_PSTT2onepass} matches~\cref{eq:M_SSTT}, leading to neither algorithm being better in storage.
\end{itemize}
In practice, all our algorithms require some extra memory for the communicated quantities. However, this cost depends on $1/P$, which is by assumption less than $1/C$, and is thus absorbed in our asymptotic statements.

\begin{figure}
  \centering
  \includegraphics[width=0.32\textwidth]{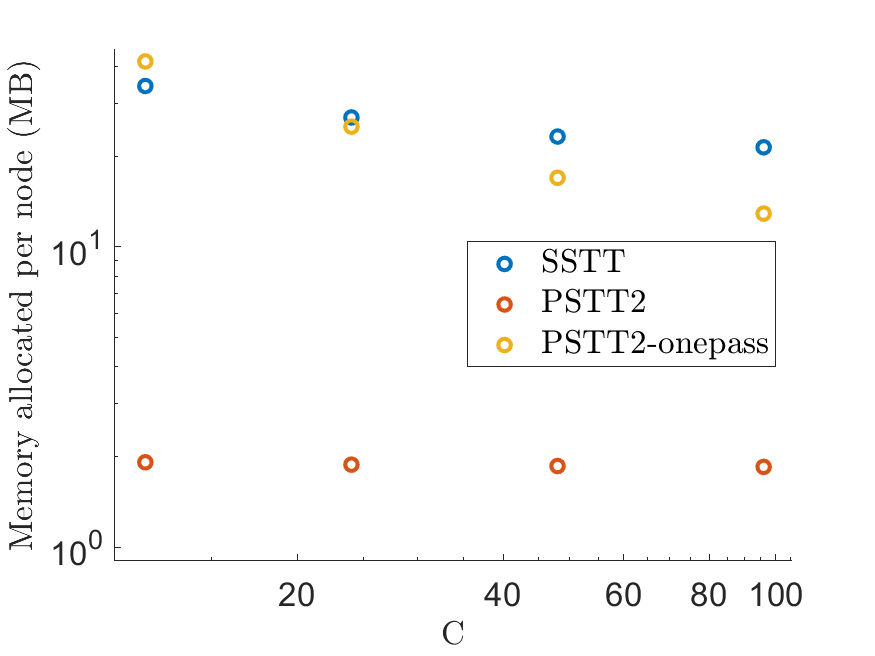}
  \includegraphics[width=0.32\textwidth]{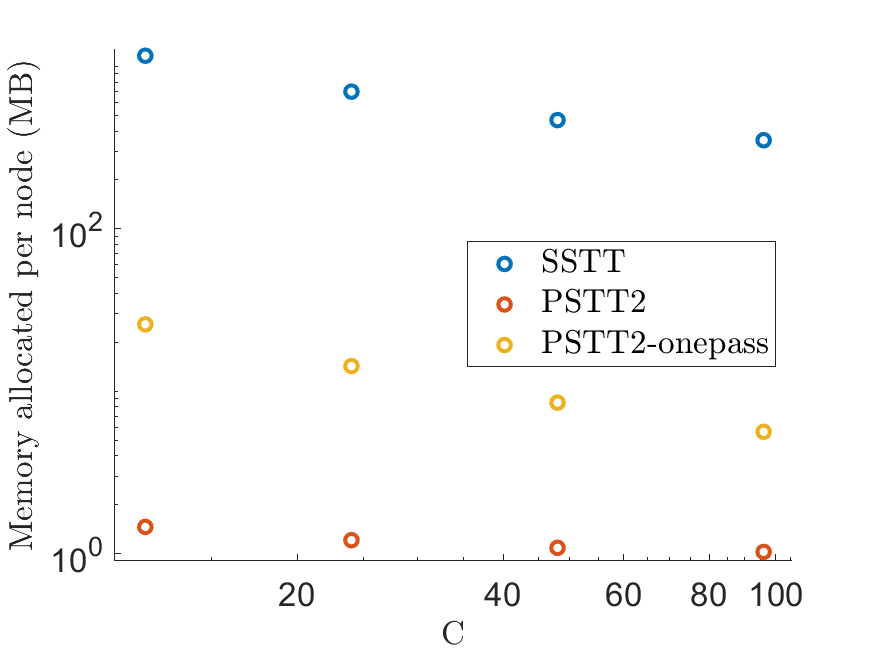}
  \includegraphics[width=0.32\textwidth]{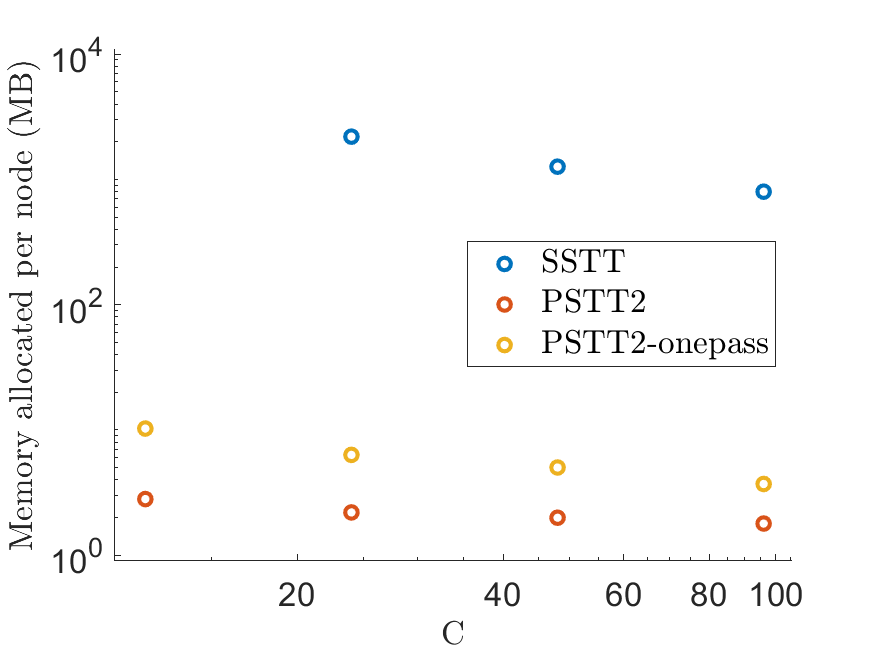}
  \caption{Measured total memory allocated per-core for (left) $d=3$, (right) $d=5$, and (bottom) $d=9$. The bottom plot is missing a single data point for SSTT with $C=12$ because the $d=9$ tensor is too large to store in memory. See~\cref{tab:1} for additional details of each experiment.}
  \label{fig:heap}
\end{figure}

In~\cref{fig:heap}, we see the total memory allocated per-core as a function of the total number of cores $C$ for the three different Hilbert tensor TT computations. For each algorithm and each Hilbert tensor, we use $C=12$, $C=24$, $C=48$, and $C=96$. We see that PSTT2 reliably has the smallest memory overhead, and the improvement over SSTT increases as the dimension increases. This agrees with the asymptotic expectations of~\cref{eq:M_PSTT2} and~\cref{eq:M_SSTT}. We see that PSTT2-onepass also improves upon SSTT for all but the $d=3$ trial, where the two asymptotic spatial complexities agree. Moreover, the $d=9,$ $C=12$ SSTT trial data is missing due to an out-of-memory error. This occurs because the first sketch has $r_1=n_1$, leading to no compression on the first step of SSTT. As a result, the first $Z$ calculated needs to store the full tensor of 38.4 GB in memory, which is over the 32 GB limit of a single node.

As a function of $C$, the PSTT2 memory profiles are relatively flat due to the importance of storing sketch matrices. Specifically, PSTT2 storage costs are more influenced by the cost of storing DRMs and other similarly sized allocations, especially for small $d$. These costs are constant per-core, whereas the sub-sketch storage costs are distributed evenly across the cores as much as possible.

\subsection{Time Complexity} \label{sec:Time_Complexity}
Finally, we move to the discussion of time complexity. The major cost of our algorithms is streaming via the function $f$, which we assume to be proportional to the size of the tensor. Because streaming is evenly split amongst the cores, the complexity to stream the tensor once is
\begin{equation}
\label{eq:Tstream}
  T_{\rm stream} = \tau_f n^d/C, 
\end{equation}
where capital `$T$'s are used to designate time complexities. From numerical experiments, we find that~\cref{eq:Tstream} takes a significant portion of the computation time in the $C=1$ setting with no communication. We also want to remark that PSTT2-onepass has half the cost of streaming since it only streams the tensor once. 

Even when the function evaluation is cheap, the time to multiply the DRMs against the sub-tensors can still be expensive. There are $d-1$ sketches for PSTT2, each with a time cost of $\mathcal{O}(r n^d/C)$, and this leads to a total sketching time of $\mathcal{O}(d r n^d/C)$. PSTT2-onepass has a single more sketch to compute, but the complexity is the same as $d$ becomes large. Comparatively, SSTT has a leading order complexity of $\mathcal{O}(r n^d/C)$, as only the first sketch and calculation of $Z$ involve multiplications with the full tensor.

\begin{figure}
  \centering
  \includegraphics[width=0.45\textwidth]{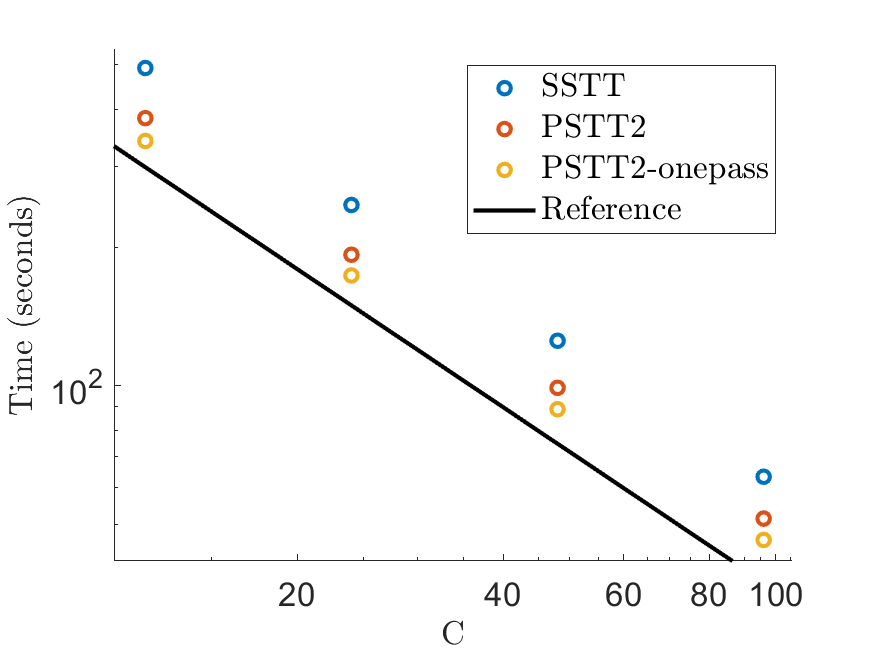}
  \includegraphics[width=0.45\textwidth]{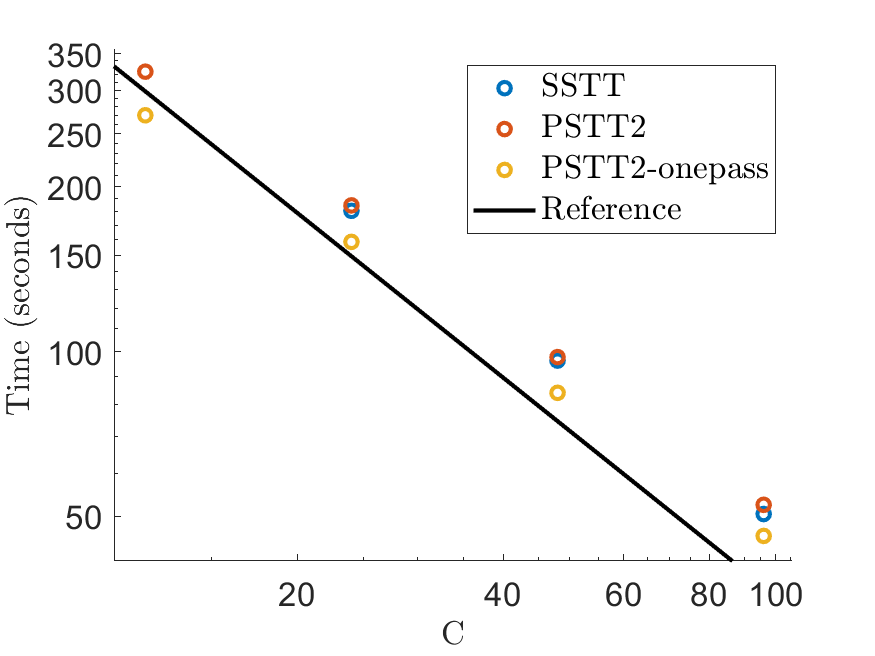}
  \caption{Strong scaling for the Hilbert tensor (left) $d=5$, and (right) (b) $d=9$. On each plot, a black line with slope $-1$ is shown for comparison (see~\cref{tab:1} for details).}
  \label{fig:time}
\end{figure}

\Cref{fig:time} shows the strong scaling timing results for the same $d=5$ and $d=9$ Hilbert tensors in the memory experiments (see~\cref{tab:1}). The slopes of the strong scaling times are near $-1$ for all algorithms and both values of $d$, i.e.~the time nearly scales as $1/C$. This agrees with both~\cref{eq:Tstream} and the scaling of multiplying against the tensor. We see that PSTT2 and SSTT take comparable amounts of time for $d=9$ and PSTT2 is moderately faster for $d=5$. However, the speedup is only by a constant factor, and is difficult to predict. In comparison, PSTT2-onepass is significantly faster than the other two algorithms, due to the single streaming loop. 

The time complexity of communication is more complicated to model. From inspection of~\cref{alg:parsketch}, we see that the multi-sketching step can have $d P$ sends and receives at worst. However, because sometimes the core that calculates $S_{i,k,j}$ also stores $S_{i,k}$, not every communication is necessary. For example, in SSTT, the first sketch can be performed with zero communication when the cores stream columns of the first unfolding. For PSTT2 and PSTT2-onepass, this strategy of streaming columns of unfoldings can also eliminate the communication for column space sketches, but the cost of row sketches remains. In conclusion, as the worst-case communication depends on the partitions $\mathbf{P}$, we can think of $\mathbf{P}$ as an appropriate balance between good memory performance and timing performance.

\begin{figure}
  \centering
  \includegraphics[width=0.45\textwidth]{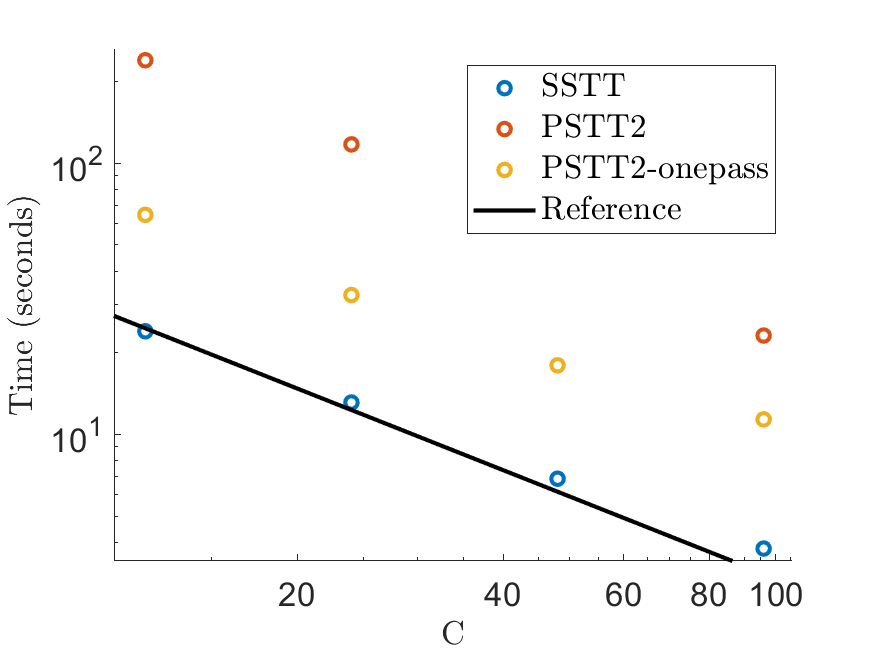}
  \includegraphics[width=0.45\textwidth]{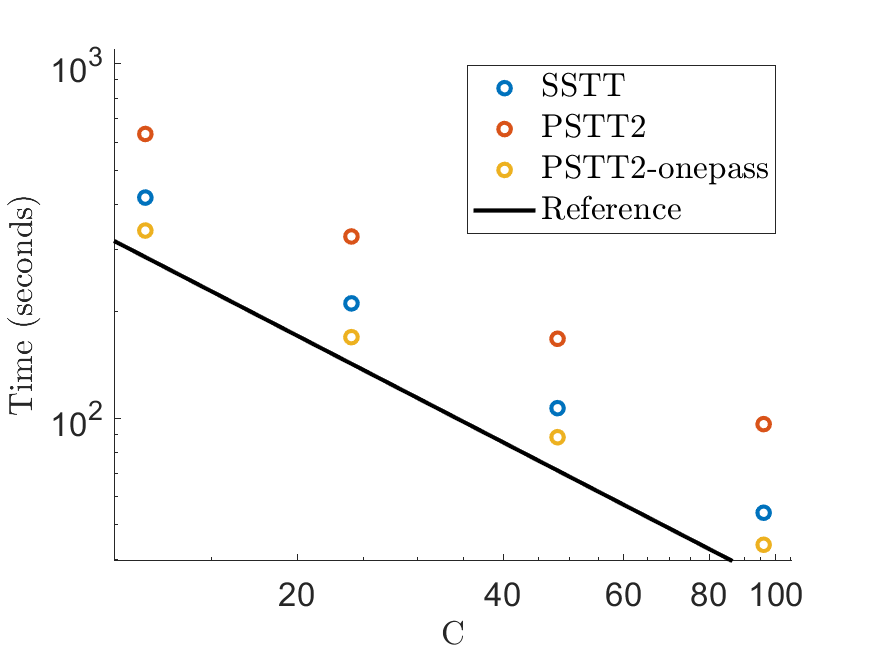}
  \caption{Strong scaling for the $d = 3$ Hilbert tensor (left) and $d=3$ Gaussian bumps tensor (right). On each plot, a black line with slope $-1$ is shown for comparison (see~\cref{tab:1} for details).}
  \label{fig:time2}
\end{figure}
\Cref{fig:time2} (left) shows that our parallel algorithms suffer from communication costs when the tensor size is small. For this $d=3$ Hilbert tensor, SSTT is significantly faster than PSTT2 and PSTT2-onepass. From~\cref{tab:1}, we see that the sub-tensors streamed are much smaller, leading to a smaller contribution of the total time streaming relative to communication. However, we emphasize that these communications are still relatively unimportant when the evaluation of a tensor element is expensive, as can be seen in an experiment with a Gaussian mixture model, with the function defined as
\begin{equation*}
  \mathcal{X}_{i_1,i_2,i_3} = \sum_{j = 1}^N e^{-\gamma \left((x_1-\xi_j)^2 + (x_2-\eta_j)^2 + (x_3 - \zeta_j)^2
  \right)}, \qquad x_j = \frac{2i_j}{n_j}-1,
\end{equation*}
where $\xi_j$, $\eta_j$, and $\zeta_j$ are the center coordinates of the $j$th Gaussian and $\gamma$ controls the width of the Gaussians.~\Cref{fig:time2} (right) shows the result when $N=100$, $\gamma=10$, and the centers are uniformly random numbers between $-1$ and $1$. Because each calculation of a tensor entry takes the evaluation of $N$ Gaussians, the evaluation is significantly longer than that of a Hilbert tensor. We see that although PSTT2 is still slower than SSTT, PSTT2-onepass becomes the fastest option since avoiding the second set of function evaluations gains more than the added communications.

\section{Solve Sylvester tensor equations in TT format} \label{sec:TTsylv}
Many tensors in practice are known implicitly as solutions of tensor equations. For example, the discretized solution of a multivariable PDE might satisfy a tensor equation.

In this section, we focus on computing an approximation $\tilde{\mathcal{X}}$ in TT format of a tensor $\mathcal{X} \in \C^{n_1 \times \dots \times n_d}$ that satisfies the Sylvester tensor equation in~\cref{eq:TensorDisplacement}. This type of algebraic relation appears when discretizing Poisson's equation of a tensor-product domain with either a finite difference scheme~\cite{leveque2007finite} or a spectral method~\cite{shi2021compressibility}. We describe an algorithm that solves the 3D Sylvester tensor equation of the form
\begin{align}
\mathcal{X} \times_1 A + \mathcal{X} \times_2 B + \mathcal{X} &\times_3 C = \mathcal{F}, \quad \mathcal{F} \in \C^{n_1 \times n_2 \times n_3}, \nonumber \\
A, B, C \ {\rm normal}, \quad A \in \C^{n_1 \times n_1}&, \quad B \in \C^{n_2 \times n_2}, \quad C \in \C^{n_3 \times n_3}. \label{3d:sylv}
\end{align}
To ensure a unique solution, we assume that $\lambda_i(A) + \lambda_j(B) + \lambda_k(C) \neq 0$, for $1 \le i \le n_1$, $1 \le j \le n_2$, and $1 \le k \le n_3$, 
where $\lambda_i(A)$ is an eigenvalue of $A$~\cite{simoncini2016computational}. We further suppose that $\mathcal{F}$ is given in the TT format with TT rank $(1,r_1,r_2,1)$ and cores $\mathcal{G}_1, \mathcal{G}_2$, and $\mathcal{G}_3$. Our goal is to compute an approximate solution $\tilde{\mathcal{X}}$ to~\cref{3d:sylv} in the TT format.

Since TT cores of a tensor are analogues of factor matrices in a matrix decomposition, our key idea is to convert~\cref{3d:sylv} into several Sylvester matrix equations and use fADI. If we reshape $\mathcal{X}$ and $\mathcal{F}$ to their 1st unfolding, respectively, then~\cref{3d:sylv} becomes
\begin{equation} \label{3d:sylv_fla1}
AX_1 + X_1(I \otimes B + C \otimes I)^T = F_1 = G_1(G_2)_1(G_3 \otimes I),
\end{equation}
where $G_1$ is the 1st unfolding of $\mathcal{G}_1$, $G_3$ is the 2nd unfolding of $\mathcal{G}_3$, and $(G_2)_1$ is the 1st unfolding of $\mathcal{G}_2$. Similarly, we get another Sylvester matrix equation by reshaping $\mathcal{X}$ and $\mathcal{F}$ to their 2nd unfolding:
\begin{equation} \label{3d:sylv_fla2}
(I \otimes A + B \otimes I) X_2 + X_2C^T = F_2 = (I \otimes G_1)(G_2)_2G_3,
\end{equation}
where $(G_2)_2$ is the 2nd unfolding of $\mathcal{G}_2$.

Algorithm 4.1 in~\cite{shi2021compressibility} (ST21) describes a way to compute the solution of~\cref{3d:sylv} in TT format. It starts by computing the first TT core $U_1$ as an ON basis of the column space of $X_1$ from~\cref{3d:sylv_fla1}, and then finds the other two TT cores via solving:
\begin{equation} \label{3d:sylv_fla2v} 
(I \otimes \tilde{A} + B \otimes I) Y + YC^T = (I \otimes \tilde{G}_1)(G_2)_2G_3,
\end{equation}
where $\tilde{A} = U_1^*AU_1$, and $\tilde{G}_1 = U_1^*G_1$.~\cref{3d:sylv_fla2v} is similar to~\cref{3d:sylv_fla2} but~\cref{3d:sylv_fla2v} has two disadvantages: (1) Parallelism when finding TT cores is not exploited, as the two fADI loops are carried out in an order, (2) $\tilde{A}$ is a dense matrix, which may have a higher cost of solving shifted linear systems than that of $A$. In this way, we want an algorithm that, in certain scenarios, can solve~\cref{3d:sylv_fla1} and~\cref{3d:sylv_fla2} simultaneously and only solve shifted linear systems with $A, B$, and $C$.

With fADI, we can obtain $U_1$ and $U_2$, ON bases of the column space of $X_1$ and $X_2$, from~\cref{3d:sylv_fla1} and~\cref{3d:sylv_fla2} independently. In the meantime, the third TT core can be computed along with $U_2$. As a result,~\cref{thm:interlacing_d} allows us to compute all three TT cores with only one extra matrix-matrix multiplication.

In general, we need two sets of shift parameters to solve~\cref{3d:sylv_fla1} and~\cref{3d:sylv_fla2}. The Zolotarev number associated with~\cref{3d:sylv_fla1} is
\begin{equation} \label{zolo_fla1}
Z_k(\Lambda(A), \Lambda(-B)+\Lambda(-C)) := \inf_{r \in \mathcal{R}_{k,k}} \frac{\sup_{z \in \Lambda(A)} |r(z)|}{\inf_{z \in \Lambda(-B)+\Lambda(-C)} |r(z)|},\qquad k\geq 0,
\end{equation}
where `+' is the Minkowski sum of two sets. Similarly, the Zolotarev number corresponded to~\cref{3d:sylv_fla2} is
\begin{equation} \label{zolo_fla2}
Z_k(\Lambda(A)+\Lambda(B), \Lambda(-C)) := \inf_{r \in \mathcal{R}_{k,k}} \frac{\sup_{z \in \Lambda(A)+\Lambda(B)} |r(z)|}{\inf_{z \in \Lambda(-C)} |r(z)|},\qquad k\geq 0.
\end{equation}
We find that the number
\[
L_k(\Lambda(A), \Lambda(B), \Lambda(C)) := \inf_{r \in \mathcal{R}_{k,k}} \frac{\sup_{z \in \Lambda(A) \cup [\Lambda(A)+\Lambda(B)]} |r(z)|}{\inf_{z \in \Lambda(-C) \cup [\Lambda(-B)+\Lambda(-C)]} |r(z)|},\qquad k\geq 0,
\]
is an upper bound for both~\cref{zolo_fla1} and~\cref{zolo_fla2}, so that bounds in~\cref{zolo_fro} can be acquired for both $X_1$ and $X_2$:
\[
\|X_i-(X_i)_k\|_F \le L_k(\Lambda(A), \Lambda(B), \Lambda(C)) \|\mathcal{X}\|_F, \qquad i = 1,2.
\]
Therefore, we can choose the same shift parameters when we use fADI on~\cref{3d:sylv_fla1} and~\cref{3d:sylv_fla2}, which means $U_1$ and $U_2$ can be found in a single set of iterations.

Since we do not need $U_2$ in a low-rank format, we can recover $U_2$ column-by-column in a way introduced in ST21 by using alternating direction implicit (ADI) method~\cite{benner2009adi} on reshapes of the columns. We summarize the 3D Sylvester equation solver in~\cref{alg:7}.

\begin{algorithm}
\caption{TT-fADI: Given a 3D Sylvester tensor equation~\cref{3d:sylv}, compute an approximate solution in TT format with three cores computed almost-simultaneously.}
\begin{algorithmic}[1]
\label{alg:7}
\Require {Matrices $A, B$, and $C$, TT cores $\mathcal{G}_1, \mathcal{G}_2$, and $\mathcal{G}_3$ and TT ranks $r_1$ and $r_2$ of $\mathcal{F}$, and desired accuracy $0 < \epsilon < 1$}
\Ensure {TT cores $\mathcal{H}_1, \mathcal{H}_2$, and $\mathcal{H}_3$ of an approximate solution $\tilde{\mathcal{X}}$}
\State Use the spectra of $A, B$, and $C$ to find shift parameter arrays $\pmb{p}$ and $\pmb{q}$ of length $\ell$.
\State Solve $(A-q_1I_{n_1})Z_1 = G_1$. Let $Z = Z_1$.
\State Solve $((I_{n_2} \otimes A + B \otimes I_{n_1})-q_1I_{n_1n_2})W_1 = (I_{n_2} \otimes G_1)(G_2)_2$. Let $W = W_1$.
\State Solve $(-C-\overline{p_1}I_{n_3})Y_1 = G_3^T$. Let $Y = Y_1$.
\State Let $D = (q_1-p_1)I_{r_2}$.
\For {$1 \le j \le \ell-1$}
\State Set $R_j = (q_{j+1}-p_j)Z_j$, $U_j = (q_{j+1}-p_j)W_j$, and $V_j = (\overline{p_{j+1}}-\overline{q_j})Y_j$.
\State Solve $(A-q_{j+1}I_{n_1})Z_{j+1} = R_j$. Set $Z_{j+1} = Z_{j+1}+Z_j$ and $Z = \begin{bmatrix} Z & Z_{j+1} \end{bmatrix}$.
\State Solve $((I_{n_2} \otimes A + B \otimes I_{n_1})-q_{j+1}I_{n_1n_2})W_{j+1} = U_j$. Set $W_{j+1} = W_{j+1}+W_j$ and $W = \begin{bmatrix} W & W_{j+1} \end{bmatrix}$.
\State Solve $(-C-\overline{p_{j+1}}I_{n_3})Y_{j+1} = V_j$. Set $Y_{j+1} = Y_{j+1}+Y_j$ and $Y = \begin{bmatrix} Y & Y_{j+1} \end{bmatrix}$.
\State Set $D = \begin{bmatrix} D & \\[5pt] & (q_{j+1}-p_{j+1})I_{r_2} \end{bmatrix}$.
\State Recompress $W, D$, and $Y$ to get $\|\tilde{W}\tilde{D}\tilde{Y}^*-WDY^*\| \le \epsilon \|WDY^*\|$. 
\State Set $W = \tilde{W}, D = \tilde{D}$, and $Y = \tilde{Y}$, and $s_2$ to be the rank.
\EndFor
\State Compute a CPQR of $Z$ to obtain $U_1$ with ON cols and set $U_1 = U_1(:,1\!:\!s_1)$ if $U_1(s_1+1,s_1+1) \le \epsilon$.
\State Calculate $T = U_1^*{\rm reshape}(W,n_1,n_2s_2)$.
\State Set $\mathcal{H}_1 = U_1$, $\mathcal{H}_2 = {\rm reshape}(T, s_1, n_2, s_2)$, and $\mathcal{H}_3 = DY^*$.
\end{algorithmic}  
\end{algorithm} 

We demonstrate~\cref{alg:7} with a simple example. Consider the equation
\begin{equation} \label{sylv_sim_ex} 
\mathcal{X} \times_1 A + \mathcal{X} \times_2 A + \mathcal{X} \times_3 A = \mathcal{F}, \quad A \in \R^{n \times n}, \ \mathcal{F} \in \R^{n \times n \times n}, 
\end{equation}
where $A$ is a diagonal matrix with diagonal elements $a_j \in \left[-1, -1/(30n) \right]$ for $1 \le j \le n$, and $\mathcal{F}$ has TT rank $(1,\floor{n/4},2,1)$ with all three TT cores consisting of i.i.d. uniform random numbers in $(0,1)$.~\Cref{fig:3dsylv_ex} shows the running time of three Sylvester equation solvers. The green line represents the algorithm ST21. The blue line represents a direct solver, which computes each element of $\mathcal{X}$ by $\mathcal{X}_{i,j,k} = \mathcal{F}_{i,j,k}/(a_i+a_j+a_k)$ for $1 \le i,j,k \le n$, and performs TT decomposition on $\mathcal{X}$. This algorithm has complexity $\mathcal{O}(n^3)$. The red line represents~\cref{alg:7}. We can see when $n \ge 100$,~\cref{alg:7} is the fastest. With $n = 350$,~\cref{alg:7} is 4 times faster than ST21, and almost 6 times faster than the direct solver. The performance of ST21 is affected since $s_1$, the size of the first TT core of the solution $\mathcal{X}$, can be close to $n$ despite the fact that $s_1=\mathcal{O}(\log n)$~\cite{shi2021compressibility}. Therefore, the cost of solving shifted linear systems of $\tilde{A}$ in~\cref{3d:sylv_fla2v} is significantly higher than that of $A$.

\begin{figure}
\centering
\begin{minipage}{0.49\textwidth}
\begin{overpic}[width=\textwidth]{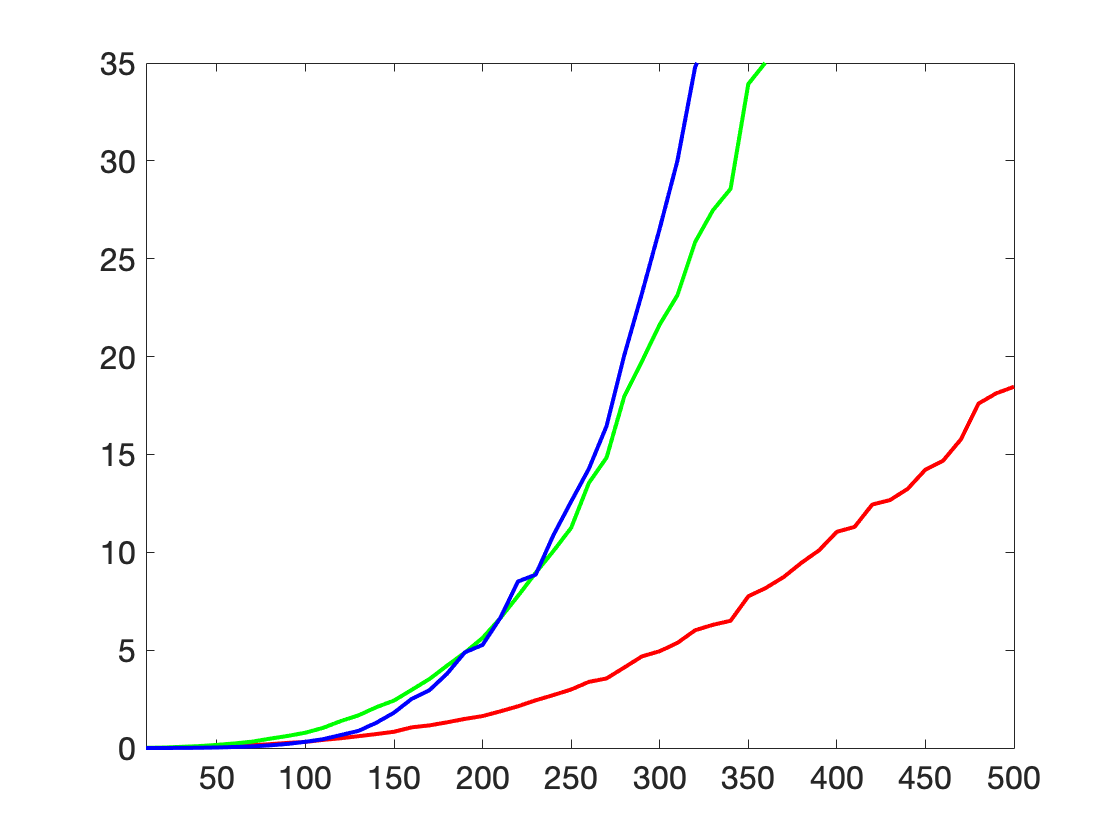}
\put(45,-2) {Size, $n$}
\put(-1,28) {\rotatebox{90}{Time (sec)}}
\put(64,12) {\rotatebox{35}{TT-fADI}}
\put(60,40) {\rotatebox{63}{ST21}}
\put(48,42) {\rotatebox{70}{Direct}}
\end{overpic}
\end{minipage}
\caption{The execution time of direct solver (blue), ST21 (green), and~\cref{alg:7} (red) to solve~\cref{sylv_sim_ex} with size of the problem $10 \leq n\leq 500$.}
\label{fig:3dsylv_ex} 
\end{figure}
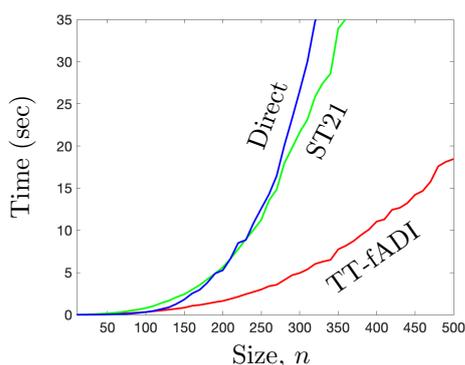

\section*{Acknowledgements}
We thank David Bindel for both his advice and his class, where the first two authors learned much about parallel computing.

\bibliography{references}
\bibliographystyle{siam}

\end{document}